\documentclass[12pt, oneside,english,reqno]{amsart}

\usepackage{fourier}

\usepackage[T1]{fontenc}
\usepackage[latin9]{inputenc}
\usepackage{amsthm}
\usepackage{amstext}
\usepackage{amssymb}
\usepackage{amsmath}

\usepackage{graphicx}
\usepackage{color}
\usepackage{url}

\usepackage{nicefrac}
\usepackage{wrapfig}

\def\bl{\begin{lemma}}
\def\el{\end{lemma}}
\def\bth{\begin{theorem}}
\def\eth{\end{theorem}}
\def\bc{\begin{corollary}}
\def\ec{\end{corollary}}
\def\bcj{\begin{conjecture}}
\def\ecj{\end{conjecture}}
\def\bpr{\begin{proposition}}
\def\epr{\end{proposition}}
\def\bde{\begin{definition}}
\def\ede{\end{definition}}
\def\E{\mathbb{E}}

\newcommand{\be}{\begin{eqnarray}}
\newcommand{\ee}{\end{eqnarray}}
\newcommand{\eps}{{\mbox{$\epsilon$}}}
\newcommand{\R}{{\mathbb R}}

\newcommand{\Z}{{\mathbb Z}}

\newcommand{\F}{{\mathcal F}}

\newcommand{\A}{{\mathcal A}}
\newcommand{\B}{{\mathcal B}}

\newcommand{\RR}{{\mathcal R}}

\newcommand{\EE}{{\mathcal E}}

\newcommand{\with}{\hbox{ {\rm with} }}
\renewcommand{\and}{\hbox{ {\rm and} }}

\newcommand{\off}{\hbox{ {\rm off} }}
\newcommand{\on}{\hbox{ {\rm only on} }}

\newcommand{\C}{{\mathcal{C}}}
\newcommand{\prob}{\mbox{\bf P}}

\newcommand{\Ss}{\mathcal{T}}

\newcommand{\lr}{\leftrightarrow}

\newcommand{\lrlong}{\longleftrightarrow}

\def\arrowfillCS#1#2#3#4{%
   \thickmuskip0mu\medmuskip\thickmuskip\thinmuskip\thickmuskip
   \relax#4#1\mkern-7mu%
   \cleaders\hbox{$#4\mkern-2mu#2\mkern-2mu$}\hfill
   \mkern-7mu#3
}
\def\leftrightarrowfillCS{\arrowfillCS\leftarrow\relbar\rightarrow\relax}

\newtheorem{theorem}{Theorem}
\newtheorem{definition}{Definition}[section]
\newtheorem{lemma}{Lemma}[section]

\newtheorem{claim}[lemma]{Claim}

\newtheorem{corollary}[lemma]{Corollary}
\newtheorem{proposition}[theorem]{Proposition}
\newtheorem{conjecture}[theorem]{Conjecture}

\newtheorem*{full}{Theorem \ref{mainthm}}

\theoremstyle{definition}
\numberwithin{equation}{section}

\DeclareMathOperator{\sign}{sign}

\gdef\SetFigFontNFSS#1#2#3#4#5{} 
\allowdisplaybreaks[4] 


\usepackage{babel}

\begin{document}
\title{Arm exponents in high~dimensional~percolation}
\author{Gady Kozma} \author{Asaf Nachmias}

\begin{abstract} We study the probability that the origin is connected to the sphere of radius $r$ (an {\em arm} event) in critical percolation in high dimensions, namely when the dimension $d$ is large enough or when $d>6$ and the lattice is sufficiently spread out. We prove that this
probability decays like $r^{-2}$. Furthermore, we show  that the probability of having $\ell$ disjoint arms to distance $r$ emanating from the vicinity of the origin is $r^{-2\ell}$.
\end{abstract}

\maketitle
%
\section{{\bf  Introduction}}

It is widely believed that there is no infinite component almost
surely in critical percolation on any $d$-dimensional lattice for
any $d > 1$. Proving this is considered one of the most challenging
problems in probability. This was proved for $d=2$ by Harris \cite{Harris} and Kesten \cite{K80} and in {\em high dimensions} by Hara and Slade \cite{HaS0}. By {\em high dimensions} we mean one of the two underlying graphs: (i) $\Z^d$ with $d \geq 19$ or, (ii) the graph with vertex set $\Z^d$ such that $x$ and $y$ are neighbors iff $|x-y|\leq L$ for sufficiently large $L$ and $d>6$ (see further definitions below).

Having no infinite component almost surely is equivalent to the assertion that the probability that the origin is connected by an open path to
$\partial Q_r$, the boundary of the cube $\{-r,\dots,r\}^d$ tends to
$0$ as $r\to \infty$. Physicists' lore (see for example \cite{AS}, page 31)
maintains that not only is there no infinite component for any $d
\geq 2$, but also that these probabilities decay according to some
power law in $r$, that is $\prob _{p_c} ( 0 \lr \partial Q_r  ) = r^{-1/\rho + o(1)}$ for some critical exponent $\rho >0$ which depends only on the dimension $d$, and not on the local structure of the lattice. In this paper we prove that $\rho=1/2$ in high dimensions.
\begin{theorem} \label{mainthm}
Consider critical percolation in high dimensions. Then we have
$$ \prob _{p_c} \big ( 0 \lr \partial Q_r \big ) \approx r^{-2} \, ,$$
\end{theorem}
\noindent Here and below, $f(r) \approx g(r)$ means that for some
constant $C>0$ which might depend on the dimension $d$ and on the
specific lattice chosen, but not on $r$, we have $C^{-1} f(r) \leq
g(r) \leq C f(r)$ for all $r>0$.
A one-arm exponent was established in a few cases in the past.
\begin{itemize}
\item Kolmogorov \cite{Ko} studied critical Galton-Watson processes and showed that for
critical percolation on an infinite regular tree, $\rho=1$ (this
can be considered as the $d=\infty$ case).
\item In the breakthrough work of Lawler, Schramm and Werner \cite{LSW}, who relied on the
work of Smirnov \cite{Sm, BR}, it is shown that $\rho=48/5$ for the
triangular lattice in two dimensions.
\item Van der Hofstad, den Hollander and Slade show that $\rho=1$ in the setting of critical {\em oriented} spread-out percolation in dimension larger than $4$.
\end{itemize}



Even though most critical exponents for high dimensional percolation
are known, the value of $\rho$ has remained undetermined. A previous
attempt at calculating $\rho$ was made by Sakai \cite{S}. He proved a
conditional result implying that $\rho=1/2$, but unfortunately his
assumptions are not known to hold. One of his assumptions is that
$\rho$ is well defined --- an assertion we do not know how to prove without employing the full mechanism of this paper.

Rigorous results about critical percolation in high dimensions were obtained using the {\em lace expansion}, a perturbative
technique inspired by the non-rigorous renormalization group
methods used by physicists. We will liberally apply results achieved using the lace expansion, described below, but we do not use this technique directly.



\subsection{Critical percolation in high dimensions.} \label{critperc} For an infinite graph $G$ and $p \in [0,1]$ we write $\prob_p$ for the probability measure on subgraphs of $G$ obtained by independently retaining each edge with probability $p$ and deleting it with probability $1-p$. Edges retained are called {\em open} and edges deleted are called {\em closed}. The critical percolation probability $p_c$ is defined by
$$ \inf  \big \{ p \, : \, \prob_p ( \exists \textrm{ an infinite component} ) > 0 \big \} \, .$$

In this paper we consider critical percolation in high dimensions. By that we mean that $G=(V,E)$ is one of the following.
\begin{itemize}
\item The {\em nearest neighbor} model with $d \geq 19$, in which the
  vertex set $V=\Z^d$ and $E = \{ (x,y) : ||x-y||_1 = 1 \}$ or,
\item The {\em spread-out} model with $d>6$, in which $V=\Z^d$  and $E = \{ (x,y) : ||x-y||_1 \leq L \}$ for some sufficiently large $L>L_0(d)$.
\end{itemize}
Informally, in high dimensions the space available for the critical percolation cluster to expand is so large, that the interactions between different parts of the cluster become negligible, forming some independence between the different parts of the cluster. When the underlying graph is an infinite regular tree, this statement can be made completely formal. Indeed, the status of the edges descending from one branch of the root is independent of the status of the edges descending from another branch. Even though such strong independence does not hold in critical percolation on $\Z^d$, we still expect the same rough behavior when $d$ is large. One formal aspect of this heuristic, is that we expect that the critical exponents, which describe the ``shape'' of the clusters, attain the same value they do on an infinite regular tree.

A fundamental result in this spirit is due to Barsky and Aizenman
\cite{BA} and Hara and Slade \cite{HaS0}. It states that in high
dimensions we have
\be\label{deltais2} \prob_{p_c} ( |\C(0)| > n ) \approx \frac{1}
{\sqrt{n} } \, ,\ee where $\C(0)$ denotes the connected
cluster containing the origin. It is a classical fact \cite{AthNey} that the same statement holds for critical percolation on an infinite regular
tree. We remark that in \cite{HaS} the precise asymptotic behavior of $\prob_{p_c} ( |\C(0)| \geq n )$ in high dimensions was obtained.

This appearance of ``tree-like'' behavior once the dimension is large occurs in many models of statistical physics. The dimension this transition occurs at is sometimes called the {\em upper critical dimension}. It is believed that for critical percolation, the upper critical dimension is $6$. In particular, it is believed that (\ref{deltais2}) holds whenever $d>6$,
however, this was proved only for the spread-out model and proving this in the full generality is still open.

In this paper we use the estimate (\ref{deltais2}) to prove our main theorem. Note, however, that we cannot expect
Theorem \ref{mainthm} to hold assuming only (\ref{deltais2}) since in
an infinite regular tree we have that (\ref{deltais2}) holds but
$\rho=1$. At first, having $\rho=1/2$ in high dimensions may seem
contradictory to the tree-like behavior mentioned above, but in fact,
$\rho=1$ in a tree corresponds to $\rho'=1$ in high dimensions, where
$\rho'$ is the {\em intrinsic} metric one-arm exponent. See \cite{KN} for more details.

The second estimate that we use, derived by Hara \cite{Ha} (for the nearest-neigh\-bor model) and by Hara,
van der Hofstad and Slade \cite{HaHS} (for the spread-out model)
states that in high dimensions \be\label{tpt}\prob_{p_c}(0\lr
x)\approx |x|^{2-d} \, ,\ee
where $0 \lr x$ denotes the event that $0$ is connected to $x$ with an
open path. In fact, in \cite{BA} it is shown that this estimate
implies (\ref{deltais2}). We may now state a more exact version of our
result
\begin{full}[conditional version] Assume $(\Z^d,E)$ is a lattice
  satisfying
\begin{enumerate}
\item $d>6$;
\item The estimate (\ref{tpt}); and
\item The edge set $E$ is invariant under reflections and coordinate
permutations.
\end{enumerate}
Then
\[ \prob _{p_c} \big ( 0 \lr \partial Q_r \big ) \approx r^{-2} \, ,\]
\end{full}
\noindent We explain why (iii) is needed in \S
\ref{subsec:connprob}.

\subsection{Outline of the proof}\label{sec:outline}
We use an induction scheme, not unlike the one used in \cite[\S
3.2]{KN} for calculating the {\em intrinsic} one-arm exponent. Let us
describe it in the roughest possible terms. Define
$\gamma(r)=\prob(0\lr \partial Q_r)$. Assume that the event $0\lr
\partial Q_{3r}$ occurred. Then one the the following must have
happened.
\begin{enumerate}
\item The cluster might have been not too small, that is,
  $|\C(0)|\ge \frac{1}{100}r^4$. By (\ref{deltais2}) this probability is at most $c/r^2$.
\item For every $j\in[r,2r]$, define
\begin{equation}\label{eq:defXj}
X_j=|\{x\in\partial Q_j:0\stackrel{Q_j}{\lrlong}x\}|\, ,
\end{equation}
 where by $0\stackrel{Q_j}{\lrlong}x$ we mean that $0$ is connected to $x$ with an open path which resides in $Q_j$.
  The second possibility is that for some $j\in[r,2r]$ we have that
  $X_j\leq r^2$. For this to happen we must have that $0$ is connected
  to $\partial Q_j$, which occurs with probability at most
  $\gamma(r)$, and then at least one $x \in \partial Q_j$ with
  $0\stackrel{Q_j}{\lrlong}x$ must be connected to $\partial Q_{3r}$,
  which costs us another $\gamma(r)$. Thus, the probability of this
  event is at most $r^2\gamma(r)^2$.
\item\label{enu:prob} The remaining case is that $X_j\ge r^2$ for
  all $j\in[r,2r]$ and $|C_0|\le \frac{1}{100}r^4$. Heuristically, if
  $X_j\ge r^2$ for all $j\in[r,2r]$ then
  we expect $|\C(0)|$ to be of size at least $r^4$. So the probability that
  $|\C(0)|$ is at most $ \frac{1}{100}r^4$ should be small, say at
  most $\frac{1}{20}$. Remembering that we also need for $0$ to be
  connected to $\partial Q_r$ we get that the probability of this
  possibility is at most $\frac{1}{20}\gamma(r)$.
\end{enumerate}
All this gives the heuristic relation
\[
\gamma(3r)\le \frac{c}{r^2}+r^2\gamma(r)^2+\tfrac{1}{20}\gamma(r) \, ,
\]
from which it is possible to prove inductively that
$\gamma(r)<C/r^2$. This is indeed the case, though we left out from
this simplified sketch several additional parameters required for the
induction to work. See the details in chapter \ref{sec:mainproof} below, starting with Lemma \ref{lem:recurs}.


The estimate of (iii) is the hardest part. Let us therefore state it as a separate result. For this we need to introduce the following random variable. For $j\in [r,2r]$ and an integer $L\in[0,r]$ we define
\be\label{eq:defAj}
A_j = \left | \big \{ y \in Q_{j+L} \setminus Q_j : 0 \lr y
\big \} \right | \, .
\ee
Recall also the definition of $X_j$ at (\ref{eq:defXj}).
\begin{theorem} \label{lowmass} There exists a constant $c>0$ such that for any $j$ sufficiently large, and any $L \geq j^{1/10}$ we have
$$ \prob_{p_c} \big ( X_j \geq L^2 \and A_j \leq cL^4 \big ) \leq (1-c) \prob_{p_c} \big (0 \lr \partial Q_j \big ) \, .$$
\end{theorem}
\noindent The exponent $1/10$ in the condition $L\ge j^{1/10}$ is immaterial and can be replaced with any positive number, however, this is unimportant since we apply this theorem with $L$ quite close to $j$.

Let us sketch the proof of Theorem \ref{lowmass}. We condition on
$X_j$ and then show using a second-moment method that
$\prob_{p_c}(A_j>cL^4\,|\,X_j)>c$. The main difficulty in the approach is the
lower bound of the conditional first moment. Heuristically, each $x\in\partial
Q_j$ ``branches out'' to $L^2$ vertices on average, so we should
have $\E(A_j\,|\,X_j>L^2)\ge L^4$, as long as
\begin{enumerate}
\item The conditioning on $X_j$ does not alter significantly the
  behavior of one $x$; and
\item the different branches coming out of every $x$ do not intersect
  too much.
\end{enumerate} A natural approach
to showing a claim of this sort would have been using the
{\em triangle condition}. See \cite{AN,BA,N87,S,KN} for details about the triangle condition and its applications. We could not make the triangle condition work directly, so we replaced it with a {\em regularity}
analysis, which is similar in spirit, even if very different in
detail. Let us expand on this topic.


\newcommand{\byebye}
{
We are left with the problem of justifying the
independence-from-the-past heuristic. Usually one uses the triangle
condition,
\[
\Delta=\sum_{x,y\in \Z^d}\prob(0\lr x)\prob(0\lr y)\prob(x\lr y)<\infty\,.
\]
Notice that when $d>6$, the triangle condition follows from the
two-point function estimate (\ref{tpt}). To see examples of how to use
$\Delta <\infty$, check \cite{AN,BA,S,KN} ***any more?***. Here we
could not make it work. We are not sure why. Thus the main novelty in
the proof is the workaround for the problems we had with applying
$\Delta<\infty$. We believe that the intermediate steps might be
useful in other situations (in fact, we devised them for a different
problem originally). To describe them we need some definitions.
Given integers
$r>s>0$, and a vertex $x\in \partial Q_r$ we define
$$
Y_{s,r}(x) =
  \Big | \Big \{ u \in x+Q_s : x \stackrel{Q_r}{\lrlong} u \Big \} \Big | \,
,$$


\begin{theorem} \label{shafan} There exists a constant $c>0$ such that for any integers $r>s>0$
and any $w \in Q_r$ and $x \in \partial Q_r$ we have that
$$ \prob \big ( Y_{s,r}(x) \geq s^4 \log^{5} s \,|\, w
  \stackrel{Q_r}{\lrlong} x \big ) \leq 2 e^{-c\log^3 s} \, , $$

\end{theorem}

%

The following definitions play a crucial role. For $x \in
\partial Q_r$ we write $\C_{Q_r}(x)$ for the set of vertices $z$ in
$Q_r$ such that $z \stackrel{Q_r}{\lrlong} x$. We say that $x$ is
$(s,r)$-bad if
$$ \prob \Big ( Y_{s,r}(x) \geq s^4 \log^5 s \, \mid \, \C_{Q_r}(x)
\Big ) \geq e^{-\log^2 s} \, ,$$ and note that we slightly abuse
the notation since this is a property of $\C_{Q_r}(x)$ rather than
$x$'s. We say that $x$ is {\bf $K,r$-bad} if there exists $s \geq
K$ such that $x$ is $(s,r)$-bad. For our final theorem we define
\be\label{xr} X_r = \big | \big \{ x \in \partial Q_r \, : \, 0
\stackrel{Q_r}{\lrlong} x \big \} \big \} \, ,\ee and
\be\label{xrsbad} X_r^{K\textrm{-bad}} = \big | \big \{ x \in
\partial Q_r \, : \, 0 \stackrel{Q_r}{\lrlong} x \and x \hbox{ is
$K,r$\textrm{-bad}} \big \} \big \} \, .\ee

\begin{theorem} \label{exploreshafan} There exists absolute constants $C>c>0$ such that for any integers $r>K>0$, any integer $M$ and any $\lambda\in(Ce^{-c\log^{4}K},1)$ we have
$$ \prob \Big ( X_{r} > M \and X_r^{K\textrm{-bad}} >\lambda X_{r} \Big ) \leq C \exp \left (-c\lambda^{2}e^{-C\log^{4}K}M\right) \, .$$
\end{theorem}

}



\newcommand{\getthisoutahere}
{
One should think about theorems \ref{shafan}
and \ref{exploreshafan} as a kind of ``$\Delta<\infty$
replacement for inelegant situations''. We believe it is instructive
to compare compare the proof of Theorem
\ref{pointsonsphere} in this paper to Theorem 1.2(i) in
\cite[\S3.1]{KN}. That theorem claims as follows. Denote
\[
G^\mathrm{eleg}(r)=\E|\{x:0\stackrel{r}{\lr} x\}|
\]
where $\stackrel{r}{\lr}$ means that the two vertices are connected by
an open path of length $\le r$.
Then $G^\mathrm{eleg}\le cr$. Very roughly, the proof applies the
Aizenman-Newman ``off technique'' to get the following inequality
\begin{multline*}
G^\mathrm{eleg}(2r) \ge \frac{c}{r}\bigg(G^\mathrm{eleg}(r)^2 -\\
  \sum_{\substack{x,y,u,v\\ |v-x|>100}}\prob(0\stackrel{r}{\lr}u)
    \prob(u\lr x)\prob(u\lr v)\prob(x\lr v)\prob(v\stackrel{r}{\lr}y)
\bigg).
\end{multline*}
In vague terms we may describe the sum as the ``interaction factor''
bounding the effect of the ``past'' on the ``future''. We want to
bound this factor. We first sum over $y$ and get a factor of
$G^\mathrm{eleg}(r)$. We then sum over $v$ and $x$ and get, if only
$100$ is sufficiently large, something less than
$\frac{1}{2}$. Summing over $u$ then gives another factor of
$G^\mathrm{eleg}(r)$ and we see that the interaction factor is indeed
negligible, we get $G^\mathrm{eleg}(2r)\ge (c/r)G^\mathrm{eleg}(r)^2$
from which we can conclude the claim. See \cite{KN} for all the
details.

Let us try the same approach for the problem here. Define
\[
G^\mathrm{ineleg}(r)=\E|\{x\in\partial
Q_r:0\stackrel{Q_r}{\lrlong}x\}|
\]
and repeat the same approach. We reach a very similar inequality,
\begin{multline*}
G^\mathrm{ineleg}(2r) \ge c\bigg(G^\mathrm{ineleg}(r)^2 -\\
\sum_{\substack{x\in \partial Q_r, y\in \partial Q_r(x), u\in Q_r\\
    v\in Q_r(x), |v-x|>100}}
\prob(0\stackrel{Q_r}{\lrlong}u)\prob(u\stackrel{Q_r}{\lrlong}x)
\prob(u\lr v)\prob(x\lr v)\prob(v\stackrel{Q_r(x)}{\lrlong}y)
\bigg).
\end{multline*}
Now, when we try to sum the interaction we term we face a problem: $v$
is not on the boundary of $Q_r$ but inside it. Hence the sum needs to
be done differently, and here is where Theorem \ref{shafan} comes into
the play. We use it to sum over the $v$ while keeping $x$ fixed. The
full details are found below. Here we conclude by suggesting to the
readers to keep in mind the following receipe: suppose you know of a
problem similar to your own which has a proof using only
$\Delta<\infty$ (``an elegenat proof''). Try to apply the elegant
approach to your problem. If you fail because of
issues of symmetry and order of summation, apply Theorem \ref{shafan}
or \ref{exploreshafan} to force the sum to converge in the needed order.
}




\subsection{Connection probability and cluster regularity}\label{subsec:connprob}

A key element in the proof, is a lower bound on the
connection probabilities. Let us state it formally.
\begin{lemma}\label{connect}
Let $\mathbb{Z}^{d}$ be a lattice in $\mathbb{R}^{d}$ such that
the edge set $E(\mathbb{Z}^{d})$ is symmetric with respect to coordinate
permutations and reflections. Then there exist constants $C>0$ and $c>0$ such that for any $z\in\partial Q_{r}$ we
have\[ \prob_{p_c}(0 \stackrel{Q_r}{\lrlong} z)\geq
ce^{-C\log^{2}r}.\]
\end{lemma}
An interesting feature of this lemma is that it holds in {\em all   dimensions}. However, the estimate is definitely not sharp, and we believe that the probability is in fact polynomially small and that it is minimized when $z$ sits in the corner of the cube, and then the probability is $\approx r^{\xi(d)}$ with $\xi(d)=2-2d$ when $d>6$.
The proof of this lemma is elementary, and it is there that we require the lattice to be invariant under coordinate permutations and reflections.

\begin{figure}
\input{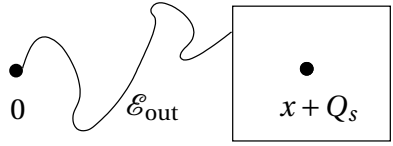}
\caption{\label{cap:eout} The event $\EE_\textrm{out}$ is independent of the edges in $x+Q_s$.}
\end{figure}
Even though it is not sharp, Lemma \ref{connect} suffices to prove regularity results on the cluster of the origin. The precise form of regularity we need is somewhat technical and we expand on that in chapter \ref{sec:ldv}. Here let us
demonstrate it with a simple example. Let $x\in\Z^d$ and define for any $A \subset \Z^d$
\[
\C(x; A)=\{y:x\stackrel{A}{\lrlong}y\}\,.
\]

Let $s\geq 0$ and consider $\C(x; x+Q_s)$. It is
well known since 1984 \cite{AN} that percolation clusters have an
exponential tail beyond their typical ``large'' size, which is $s^4$ in
our case. In other words
\be\label{eq:expo}
\prob_{p_c}(|\C(x; x+Q_s)| \ge s^4\log^3s) \le Ce^{-c\log^3 s}\,.
\ee
A regularity statement we wish to prove roughly asserts that the same bound (\ref{eq:expo}) holds even if we know that $0 \lr x$.
Formally we wish to prove that
\begin{equation}\label{eq:exmp}
\prob_{p_c}(|\C(x; x+Q_s)| \ge s^4\log^3 s\;|\;0\lrlong x) \le Ce^{-c\log^3 s} .
\end{equation}
What we need for the proof of Theorem \ref{lowmass} is somewhat different, but the idea is similar and Lemma \ref{connect} plays a crucial role.
To understand how the conditioning affects the picture, define
$\EE_\textrm{out}$ to be
the event that $0$ is connected to $x+Q_s$ (see Figure \ref{cap:eout}).
We now apply Lemma \ref{connect}. We get
\be\label{eq:E1E2}
\prob_{p_c}(0\lr x\,|\,\EE_\textrm{out})\ge ce^{-C\log^2 s} \, ,
\ee
simply because conditioning on $\EE_\textrm{out}$ reveals no
information about the status of the edges in $x+Q_s$, and it is enough
for $0\lr x$ to let $x$ connect to a single point on the boundary $x+\partial Q_s$.
The former reason also shows that the events $\EE_\textrm{out}$ and $|\C(x; x+Q_s)|\geq s^4\log^3 s$ are independent. Hence,
\begin{align*}
\lefteqn{\prob_{p_c}(0\lrlong x,\,|\C(x; x+Q_s)|\ge s^4\log^3s)\le}&&&\\
&&&\le\prob_{p_c}(\EE_\textrm{out},\,|\C(x; x+Q_s)|\ge s^4\log^3s)=\\
&&&=\prob_{p_c}(\EE_\textrm{out})\prob_{p_c}(|\C(x; x+Q_s)| \ge s^4\log^3s)\le\\
\mbox{By (\ref{eq:expo})}&&&\le \prob_{p_c}(\EE_\textrm{out}) Ce^{-c\log^3 s}\le\\
\mbox{By (\ref{eq:E1E2})}&&&
  \le \prob_{p_c}(0\lr x)
    Ce^{-c\log^3 s+C\log^2 s}\le Ce^{-c\log^3 s}\prob_{p_c}(0\lrlong x)\, ,
\end{align*} which shows (\ref{eq:exmp}).
Inequality (\ref{eq:exmp}) is a {\em local regularity} assertion. From it
one may get {\em global regularity} results in which similar estimates hold for most points of the cluster simultaneously. See
the full details in chapter \ref{sec:ldv}.

\subsection{Multiple arms.} The heuristic presented above suggests that the probability of having $\ell$ disjoint arms emanating from a small neighborhood of the origin is $r^{-2\ell}$. This is indeed the contents of the following theorem.

\begin{theorem} \label{multiplearms} For any integer $\ell\geq 1$
  there exists a constant $K=K(\ell)$ such that for any $r>0$ and for
  any vertices $y_1,\ldots, y_\ell \in B(0,\frac{1}{2}r)$ satisfying
$|y_i-y_j|\geq K$ for any $i\neq j$ we have
$$ \prob_{p_c} \big ( \{y_1 \lr \partial Q_r\} \circ \cdots \circ \{y_\ell \lr \partial Q_r\} \big ) \approx r^{-2\ell} \, ,$$
where the constants implied depend on $\ell$,$d$ and the specific lattice, but not on $r$.
\end{theorem}
\noindent The upper bound of this theorem follows immediately from the BK inequality, however, the lower bound requires an ``inverse''-BK argument.

\subsection{The BK-Reimer inequality.} We close this introduction with a remark that might be interesting
to some. In the proof of Lemma \ref{sume1e2e3} we use Reimer's version
of the van den Berg-Kesten inequality \cite{R00,BCR}. It does not seem as if the
event at hand is an intersection of an increasing and a decreasing
event, so one cannot
replace it with the simpler van den Berg-Fiebig version
\cite{BF87}. Nor did we see an obvious reduction to any simpler
inequality. In short, it seems the full power of Reimer's inequality
is needed. In a similar spirit we drop the convention
of using $\circ$ for increasing events and $\square$ for general
events, and use $\circ$ for both.

\subsection{Notations.}\label{sec:notations}
By ``lattice'' we mean a graph embedded in $\R^d$ such that the vertex
set is $\Z^d$ and the edge set, which we shall denote by $E(\Z^d)$, is
periodic with respect to a group of translations spanning $\R^d$ (by
linear combinations). We
assume the degree of each vertex is finite.
We write $Q_r\subset \Z^d$ for the cube $\{-r,\ldots,r\}^d$ and
$\partial Q_r$ for its internal boundary
$$ \partial Q_r = \big \{ z \in Q_r : \exists x \not \in Q_r
\with (z,x) \hbox{ is an edge in } \Z^d \big \} \, .$$
We will not be very strict about $r$ being an integer, and in these
cases we denote $Q_r=Q_{\lfloor r \rfloor}$ etc.


For two vertices $x,y$ we write $x \lr y$ for the event that $x$ is
connected to $y$ by an open path. It will be convenient to assume that
$x\lr x$ occurs always. We write $\C(x)$ for the connected component containing $x$, that is, $\C(x) = \{ y : x\lr y\}$.
For a subset $A \subset \Z^d$ we write $x \stackrel{A}{\lrlong} y$ for the event that $x$ is connected to $y$ by an open path which is contained in $A$ (in particular, we must have $x,y\in A$) and we write $\C(x; A)$ to denote the vertices connected to $x$ within $A$, that is $\C(x;A) = \{ y : x \stackrel{A}{\lrlong} y \}$, as define above. We say that $x \lr y$ off $A$ if there is an open path connecting $x$ and $y$ which avoids the vertices of $A$.

For two events $\A, \B$ we write $\A \circ \B$ for the event that
there exists two disjoint sets $U$, $V\subset E(\Z^d)$ such that the
status of the edges of $U$ determines $\A$, and the status of the edges  of $V$ determines $\B$. We frequently use the BK-Reimer inequality stating that $\prob(\A \circ \B) \leq \prob(\A) \prob(\B)$ (see \cite{R00,BCR}) and the FKG inequality stating that if $\A$ and $\B$ are monotone increasing, then $\prob(\A \cap \B) \geq \prob(\A) \prob(\B)$. \medskip

For $x,y\in \Z^d$ we write $|x-y|$ for the Euclidean distance
$\sqrt{\sum_i \left((x)_i-(y)_i\right)^2}$, $||x-y||_\infty$ for
$\max |(x)_i-(y)_i|$ and $||x-y||_1$ for $\sum_i|(x)_i-(y)_i|$. For a
subset of vertices $S \subset \Z^d$ and a vertex $x\in \Z^d$ we write
$x+S$ for the translation $x+S = \{ x+ s : s\in S\}$.
We denote by $c$ and $C$ positive constants which depend only on $d$ and on the specific lattice. The value of $c$ and $C$ will change from place to place, even within the same formula --- occasionally we will number
the constants $c_1,c_2,\dotsc$ for clarity. Numbered constants do not
change their value. We use $c$ for constants which are ``small
enough'' and $C$ for constants which are ``large enough''. The
notation $X\approx Y$ is short for $cX<Y<CX$. We did not
make any attempt at optimizing constants in this work. Finally, let us
remark on the use of $K$. We use $K$ consistently to denote a small
translation or a small distance between two points. In a typical
lemma, $K$ will start out as a free parameter, but will be fixed to a
constant (depending on $d$ and the lattice) when enough information
was gathered. From that point on, we will consider it as just another $C$.

\subsection{Organization}
In the next chapter we show how to formalize the heuristic
relation mentioned in \S \ref{sec:outline} and then, using Theorem \ref{lowmass}, perform the induction which yields the proof of Theorem
\ref{mainthm}.

The majority of the paper is dedicated to proving Theorem
\ref{lowmass}. In chapter \ref{sec:connprob} we prove
Lemma \ref{connect}. We use this in chapter \ref{sec:ldv} to derive
the regularity theorem, and apply all this to prove Theorem \ref{lowmass}
in chapter \ref{sec:lowmass}. We conclude by proving Theorem \ref{multiplearms} in chapter \ref{sec:multiple}.

\section{{\bf The induction scheme: proof of Theorem \ref{mainthm} using Theorem \ref{lowmass}}}\label{sec:mainproof}

In this chapter we show how to derive our main result, Theorem
\ref{mainthm} from Theorem \ref{lowmass}.
The difficulty in the proof of Theorem \ref{mainthm} is the upper
bound. Indeed, the lower bound on $\prob \big ( 0 \lr \partial Q_r
\big )$ follows from a simple second moment estimate using the
$2$-point function estimate (\ref{tpt}). This will be will be proved in lemma \ref{lower}, right after the following simple calculation.
\begin{lemma}\label{lem:xyz}
There exists a constant $C>0$ such that for any $r$ we have
\[
\sum_{x,y\in Q_r} \prob(0\lr x,0\lr y)\le Cr^6
\]
\end{lemma}
\begin{proof}
If $x$ and $y$ are connected to $0$,
then there exists $z$ such that the events $\{0 \lr z\}$, $\{z\lr x\}$
and $\{z \lr y\}$ occur disjointly (we allow the case $z=0$). This is
easy to see, and \cite[proof of theorem (6.75)]{G} gives a careful
derivation. By the BK inequality and the two-point function
estimate (\ref{tpt}) we get that
\begin{align*}
&&\sum_{x,y\in Q_r} \prob(0\lr x,0\lr y) & \leq
\sum_{x,y\in Q_r, z\in\Z^d}\prob(\{0\lr z\}\circ\{z\lr x\}\circ\{z\lr
y\}) \le \\
&\lefteqn{\mbox{by BK}}&&\le
\sum_{x,y\in Q_r, z\in\Z^d}\prob(0\lr z)\prob(z\lr x)\prob(z\lr
y)<\\
&\lefteqn{\mbox{by (\ref{tpt})}}&&<
C \sum _{x,y \in Q_r ,z\in \Z^d} |z|^{2-d} |x-z|^{2-d} |y-z|^{2-d} \,
.
\end{align*}
We estimate this sum in two parts. For $|z|\leq dr$ we simply sum
over $y$, then over $x$ and finally over $z$ to get
$$ \sum _{x,y \in Q_r ,|z|\le dr} |z|^{2-d} |x-z|^{2-d}
|y-z|^{2-d} < C r^6 \, .$$
In the other case, $|z|> dr$ then $|z|> 2|x|$ because
$|x|\le r\sqrt{d}$ and $d>6$ so
$|z-x| > |z|/2$ and $|z-y|> |z|/2$. Hence
\[ \sum _{x,y \in Q_r ,|z|\ge dr} |z|^{2-d} |x-z|^{2-d}
|y-z|^{2-d} < C r^{2d} \sum_{|z|>dr} |z|^{6-3d} <
Cr^6 \, .\qedhere\]
\end{proof}
\begin{lemma} \label{lower} There exists some constant $c>0$ such that
$$ \prob \big ( 0 \lr \partial Q_r \big ) \geq \frac{c}{r^2} \, ,$$
for all $r>0$.
\end{lemma}
\begin{proof}

Define the random variable $X$ by
$$ X = \Big | \Big \{ x \in Q_{2r} \setminus Q_r : 0 \lr x \Big \} \Big | \, .$$
By the $2$-point function estimate (\ref{tpt}) we have
$$ \E X > c r^d r^{2-d} = cr^2 \, .$$
The second moment is bounded by Lemma \ref{lem:xyz}, so $\E X^2 \leq Cr^6$. Observe that $X>0$
implies that $0 \lr \partial Q_r$ and hence we get
\[ \prob \big ( 0 \lr \partial Q_r \big ) \geq \frac{(\E X)^2}{\E X^2} \geq c r^{-2} \, .\qedhere\]
\end{proof}

We move to our main endeavor, that of proving Theorem \ref{mainthm} from Theorem
\ref{lowmass}. First we get from
Theorem \ref{lowmass} a recursive inequality for
$\prob(0\lr\partial Q_r)$. Let us state it as a lemma.
\begin{lemma}\label{lem:recurs}
Write $\gamma(r)=\prob( 0 \lr \partial Q_r )$. There exists positive constants $c_1$ and $C_1$ such that for all $\lambda\in(0,1]$ there exists $\eps_0=\eps_0(\lambda)$ such that for all $\eps\in(0,\eps_0)$ we have
\be\label{recursion} \gamma(r(1+\lambda)) \leq
\frac{C_1}{\sqrt{\eps} r^2}
 + \eps^{3/5} r^2 \gamma(r)\gamma\Big ( \frac{\lambda r}{2} \Big )
 + (1-c_1) \gamma(r) \,.
\ee
\end{lemma}
\noindent {\em Remark.} The value $\frac{3}{5}$ is somewhat arbitrary,
but the proof of Theorem \ref{mainthm} requires that it would be
larger than $\frac{1}{2}$.

\begin{proof} Let us first dispose of an uninteresting range of
parameters, the case that $\eps \le 2r^{-3}$. In this case we simply use
Barsky-Aizenman (\ref{deltais2}) and get
\[
\gamma(r(1+\lambda))\stackrel{(\ref{deltais2})}{\le}\frac{C}{\sqrt{r}}\le\frac{C}{\sqrt{\eps}r^2}
\]
and we are done (with no need to examine the other terms in
(\ref{recursion})).

Otherwise, define $L=\eps^{3/10}r$. Recall the
definitions of $X_j$ and $A_j$ (\ref{eq:defXj}), (\ref{eq:defAj})
preceding the statement of Theorem \ref{lowmass} (with the $L$ just
defined). If $0 \lr \partial Q_{r(1+\lambda)}$, then one of the
following events must occur
\begin{enumerate}
\item  $|\C(0)| \geq \eps r^4$,
\item For some $j\in[r(1+\lambda/4),r(1+\lambda/2)]$ we have $X_j\le L^2$
and $0 \lr \partial Q_{r(1+\lambda)}$,
\item For all $j\in[r(1+\lambda/4),r(1+\lambda/2)]$ we have
$X_j > L^2$ and $|\C(0)| < \eps r^4$.
\end{enumerate}
Denote these events by $\B_1$, $\B_2$ and $\B_3$ respectively.

\medskip
\noindent {\bf The term $\B_1$}. By Barsky-Aizenman (\ref{deltais2}) we bound
$$\prob(\B_1)\le \frac{C_1}{r^2\sqrt{\eps}} \, ,$$ which gives the first term in
(\ref{recursion}).

\medskip
\noindent {\bf The term $\B_2$}. We estimate $\prob(\B_2)$ using a
regeneration argument similar to the one used in \cite[eq.\ (3.8)]{KN}. Let
$j_0 \in [r(1+\lambda/4),r(1+\lambda/2)]$ be the first
$j$ for which $0<X_j \leq L^2$, and condition on $\C=\C(0;Q_{j_0})$. We get
\begin{equation}\label{eq:B2}
\prob(\B_2)=\sum_{A\;\textrm{admissable}}\prob(\C=A)
\prob(0\lr\partial Q_{r(1+\lambda)})\,|\,\C=A)
\end{equation}
where ``admissible'' means that $\prob(\C=A)>0$. If
$0\lr\partial Q_{r(1+\lambda)}$, then one of the vertices of
$\partial\C$ must be connected to $\partial
Q_{r(1+\lambda)}$ off $\C$, so we can write
\[
\prob(0\lr\partial Q_{r(1+\lambda)})\,|\,\C=A)\le
\sum_{x\in A\cap\partial Q_{j_0}}
  \prob(x\lr\partial Q_{r(1+\lambda)})\mbox{ off
  }A\,|\,\C=A).
\]
We now note that $\C(0;Q_j)$ allows to tell
whether $j=j_0$ or not --- no information from the rest of the
configuration is needed (here it is important that $j_0$ is the first
such $j$). Therefore the conditioning over
$\C(0;Q_{j_0})=A$ gives no information on the rest of the configuration
and we learn that
\begin{align*}
&&\prob(x\lr\partial Q_{r(1+\lambda)}\mbox{ off }A\,|\,\C=A)&=
\prob(x\lr\partial Q_{r(1+\lambda)}\mbox{ off }A) \\
&&&\le \prob(x\lr\partial Q_{r(1+\lambda)}) \\
&\lefteqn{\textrm{since }x\in\partial Q_{j_0}\subset Q_{r(1+\lambda/2)}}
&&\le\gamma\left(\frac{\lambda r}{2}\right).
\end{align*}
The sum over all $x$ gives a factor of at most $L^2$ by definition of $j_0$. Plugging this into (\ref{eq:B2}) gives
\[
\prob(\B_2)\le L^2\gamma\left(\frac{\lambda r}{2}\right)
\sum_{A\;\textrm{admissable}}\prob(\C=A)\le
L^2\gamma\left(\frac{\lambda r}{2}\right)\gamma(r)
\]
which is the second term in (\ref{recursion}).


\medskip
\noindent {\bf The term $\B_3$}. It is at this point that we use Theorem
\ref{lowmass}. Let us first verify the conditions of the Theorem,
namely that $j$ is sufficiently large and that $L\ge j^{1/10}$. We may
definitely assume that $r$ is sufficiently large because for small $r$
setting $C_1$ large will render the lemma true vacuously. And $j>r$. For the
second condition we recall that at the very beginning of the lemma we
assumed $\eps > 2r^{-3}$ and then
$L> (8r)^{1/10} > j^{1/10}$ (here is where we used $\lambda \le 1$ to
make sure $j\le 2r$). Hence we may apply Theorem \ref{lowmass}.
For every integer $1 \leq i \leq \frac{1}{4}\lambda\eps^{-3/10}$ let
$$
j_i = r + \tfrac{1}{4}\lambda r + iL  \in
[r(1+\tfrac{1}{4}\lambda),r(1+\tfrac{1}{2}\lambda)]\,.
$$
Let $c_2$ be the constant
from Theorem \ref{lowmass}. We define the random variable
$$ I = \Big | \Big \{ i : X_{j_i} \geq L^2 \and
                          A_{j_i} < c_2L^4 \Big \} \Big | \, .
$$
Now, if $|\C(0)|< \eps r^4$, then
\[
\left|\left\{i:A_{j_i}\ge c_2L^4\right\}\right|<
\frac{\eps r^4}{c_2 L^4}=\frac{\eps^{-1/5}}{c_2}.
\]
However, $\B_3$ implies that all $X_{j_i}\geq L^2$ and hence
\[
\B_3 \Rightarrow I >
\left\lfloor\tfrac{1}{4}\lambda \eps^{-3/10}\right\rfloor -
  c_2^{-1} \eps^{-1/5}\,.
\]
This last formula is the most interesting restriction on the exponent
$\nicefrac{3}{5}$ in the statement of the lemma. We need it here to be
less than $\nicefrac{2}{3}$ --- otherwise the term subtracted would be
bigger than the positive term rendering the estimate useless.

On the other hand, summing the estimate of Theorem
\ref{lowmass} over all $i$ gives that
\[
\E I \leq (1-c_2)\gamma(r)\tfrac{1}{4}\lambda\eps^{-3/10}
\]
and hence by Markov's inequality
\begin{align*}
&&\prob(\B_3) \le
\prob\left(I > \left\lfloor\tfrac{1}{4}\lambda \eps^{-3/10}\right\rfloor - c_2^{-1} \eps^{-1/5}\right)
&\leq \frac{(1-c_2)\frac{1}{4}\lambda\eps^{-3/10}}
{\frac{1}{4}\lambda\eps^{-3/10} - c_2^{-1}\eps^{-1/5}-1}\gamma(r) \\
&&&\le\frac{ 1-c_2}{ 1- C \lambda^{-1}\eps^{1/10} }
\gamma(r) \, ,
\end{align*}
and with $\eps$ sufficiently small, depending on $\lambda$, this is at most $(1-c_1)\gamma(r)$, say with $c_1:=\frac{1}{2}c_2$. This is the
last term in (\ref{recursion}) and the lemma is proved.
\end{proof}

\begin{proof}[Proof of Theorem \ref{mainthm}]
Let $c_1$ and $C_1$ be as in Lemma \ref{lem:recurs}. We first fix
$\lambda>0$ sufficiently small such that
\begin{equation}\label{eq:deflam}
(1+\lambda)^2\le 2,\qquad
(1-c_1)(1+\lambda)^2\le (1-\tfrac{1}{2}c_1).
\end{equation}
Next we fix $M$ so large
such that
\begin{align}
(2C_1+8\lambda^{-2}) M^{-1/11} &\leq \tfrac{1}{2}c_1 \, ,\label{areq}\\
M^{-20/11}&\le \eps_0(\lambda)\label{eq:epsgood}
\end{align}
where $\eps_0(\lambda)$ is also from the statement of Lemma
\ref{lem:recurs}.
We shall prove by induction that for any $r$ we have $\gamma(r)
\leq Mr^{-2}$. For convenience of notation, assume we wish to
prove the claim for $r(1+\lambda)$ so the induction assumption is
\[
\gamma(s)\le\frac{M}{s^2}\quad\forall s<r(1+\lambda).
\]
We now use Lemma \ref{lem:recurs} with $\eps=M^{-20/11}$ (here is where
we need (\ref{eq:epsgood})) and get
\begin{align*}
&&\gamma(r(1+\lambda))&\le
\frac{C_1}{\sqrt{\eps}r^2}
  +\eps^{3/5}r^2\gamma(r)\gamma\left(\frac{\lambda r}{2}\right)
  +(1-c_1)\gamma(r) \le \\
&\lefteqn{\mbox{inductively}}&&
\le \frac{C_1M^{10/11}}{r^2}
  +M^{-12/11}r^2\cdot\frac{M}{r^2}\cdot\frac{4M}{(\lambda r)^2}
  +(1-c_1)\frac{M}{r^2} \le \\
&\lefteqn{\mbox{by (\ref{eq:deflam})}}&&
\le\frac{M}{(r(1+\lambda))^2}\left(M^{-1/11}(2C_1+8\lambda^{-2})
  +(1-\tfrac{1}{2}c_1)\right) \le \\
&\lefteqn{\mbox{by (\ref{areq})}}&&
\le\frac{M}{(r(1+\lambda))^2}.
\end{align*}
This concludes the proof of the theorem. \end{proof}

\section{{\bf A lower bound on connection probability: proof of Lemma \ref{connect}}}\label{sec:connprob}

In this chapter we assume neither that $d>6$ nor that (\ref{deltais2})
 or (\ref{tpt}) hold.

\begin{lemma}
\label{lem:sumP>1}Let $\mathbb{Z}^{d}$ be a bounded lattice in
$\mathbb{R}^{d}$.
Then, for any $p\geq p_{c}$ and any $r>0$, \[ \sum_{z\in\partial
Q_{r}}\prob(0 \stackrel{Q_r}{\lrlong}  z)\geq1.\]
\end{lemma}
\begin{proof}
Assume the contrary and let $\eps>0$ and $r>0$ be such that
$\sum_{z\in\partial Q_{r}}\linebreak[1]\prob(0 \stackrel{Q_r}{\lrlong} z)=1-\eps$. We will show that\begin{equation} \E |\C(0)|<\infty.\label{eq:EC0finite}\end{equation} It is well known that
this implies that $p<p_{c}$ --- see \cite[eq. (3.2)]{AN} or \cite{G}.
Hence the lemma will be proved once we demonstrate
(\ref{eq:EC0finite}).

To see (\ref{eq:EC0finite}) fix an integer $n$. Let $\mathcal{X}_n$
be the collection of $n$-tuples $0=x_{1},\dotsc,x_{n}$
satisfying that
\begin{enumerate}
\item $x_{i+1}\in x_{i}+\partial Q_{r}$; \item There exist open
simple paths $\gamma_{i}$ from $x_{i}$ to $x_{i+1}$ with
$\gamma_{i}\subset x_{i}+Q_{r}$; and \item The $\gamma_{i}$ are
vertex-disjoint except at their end-points.
\end{enumerate}
By the BK inequality and translation invariance we have\[
\prob((x_{1},\dotsc,x_{n})\in\mathcal{X}_n)\leq\prod_{i=1}^{n}\prob(0\leftrightarrow
x_{i}-x_{i-1}\mbox{ in }Q_{r})\] and summing over all possible
$n$-tuples $(x_{1},\dotsc,x_{n})$ gives\[
\prob(\mathcal{X}_n\ne\emptyset)\leq\left(\sum_{z\in\partial
Q_{r}}\prob(0\leftrightarrow z\mbox{ in
}Q_{r})\right)^{n}=(1-\epsilon)^{n}\] by our contradictory
assumption.

Now fix some $z\notin Q_{nr}$. If
$0\leftrightarrow z$ then there must exist some open simple path
$\gamma:0\to z$. Define $x_{1}=0$ and then inductively $x_{i+1}$
to be the first point on $\gamma$ after $x_{i}$ in $x_{i}+\partial
Q_{r}$. Clearly this process lasts at least $n$ steps. Hence
$0\leftrightarrow z$ implies $\mathcal{X}_n\neq\emptyset$ and in
particular \[ \prob(0\leftrightarrow
z)\leq(1-\epsilon)^{n}.\] So\[ \E\!\left(\left|\mathcal{C}(0)\cap
(Q_{(n+1)r}\setminus Q_{nr})\right|\right)\leq(2(n+1)r)^{d}(1-\epsilon)^{n}\] and summing over
$n$ gives (\ref{eq:EC0finite}) and finishes the lemma.
\end{proof}
\begin{figure}
\input{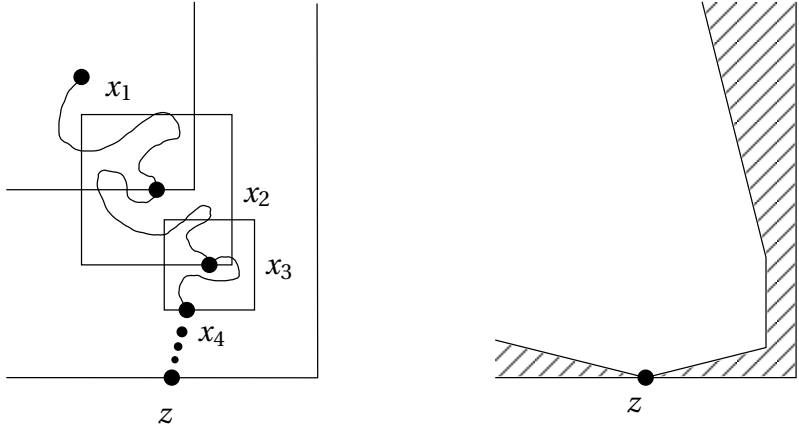}
\caption{\label{cap:z}On the left, the points $x_i$. On the right, the forbidden area for each $x_i$.}
\end{figure}
\begin{proof}[Proof of lemma \ref{connect}]
We shall construct a sequence of cubes $x_i+Q_{M_i}\subset
Q_r$, for $i=1,\dotsc, N$ and $N\leq C \log r$, such that $x_1=0$ but $x_N=z$ and
\begin{equation}\label{eq:xixi+1}
\prob(x_i\stackrel{x_i+Q_{M_i}}{\lrlong}x_{i+1})\ge cr^{1-d}.
\end{equation}
See Figure \ref{cap:z}, left.
This will of course finish
the lemma, by the FKG inequality:
\begin{align*}
&&\prob(0\stackrel{Q_r}{\lrlong}z) &
  \ge
  \prob(x_1\stackrel{Q_r}{\lrlong}{x_2},\dotsc,x_{N-1}\stackrel{Q_r}{\lrlong}x_N)\ge\\
&\mbox{by FKG}\qquad &&\ge
    \prod_{i=1}^{N-1}\prob(x_i\stackrel{Q_r}{\lrlong}x_{i+1})\ge \\
&\lefteqn{\mbox{because $x_i+Q_{M_i}\subset Q_r$}} &&\ge
    \prod_{i=1}^{N-1}\prob(x_i\stackrel{x_i+Q_{M_i}}{\lrlong}x_{i+1})\ge \\
&\lefteqn{\mbox{by (\ref{eq:xixi+1}) and $N\le C\log r$}} &&\ge
    \prod_{i=1}^{N-1}cr^{1-d}\ge ce^{-C\log^2 r}.
\end{align*}
Hence we only
need to construct the $x_{i}$.

The construction is inductive,
and it is important to keep the $x_i$ away from the boundary of $Q_r$
throughout the process --- otherwise we would not be able to choose a
reasonably big $M_i$ with $x_i+Q_{M_i}\subset Q_r$. Hence we will
require that
for every index $1\leq j\leq d$,
\begin{equation}
r-\left|\left(x_{i}\right)_{j}\right|\geq
\frac{1}{4}\left\Vert
z-x_{i}\right\Vert _{\infty}.\label{eq:farboundary}
\end{equation}
See figure \ref{cap:z}, right.

We proceed to the details of the construction. Assume $x_1,\dotsc,x_i$
have already been defined. Define
$M_{i}=\frac{1}{4}\left\Vert z-x_{i}\right\Vert _{\infty}$. By
assumption (\ref{eq:farboundary}), $x_i+Q_{M_i}\subset Q_r$, as
required. By Lemma \ref{lem:sumP>1} we have
\[
\sum_{y\in\partial Q_{M_{i}}}\prob(0 \stackrel{Q_{M_i}}{\lrlong}
y)\geq1
\]
and therefore there exist some $y$ such that
\begin{equation}
\prob(0 \stackrel{Q_{M_i}}{\lrlong} y )\geq cM_{i}^{1-d} \geq c r^{1-d} .
\label{eq:defy}\end{equation}
We want to define $x_{i+1}=x_i+y$ but that might take us in the wrong
direction, that is, not towards $z$. It is at this point that we use the symmetries of the
lattice. The symmetries allow us to rearrange the coordinates
of $y$ and change their signs and (\ref{eq:defy}) will still hold.
We do so according to the following rules:
\begin{enumerate}
\item $|y_{j}|$ are arranged like
$\left|\left(z-x_{i}\right)_{j}\right|$ i.e.\ if
$\left|\left(z-x_{i}\right)_{j}\right|>\left|\left(z-x_{i}\right)_{k\vphantom{j}}\right|$
then $|y_{j}|\geq|y_{k}|$;\label{enu:order}
\item In directions $j$ where
$\left|\left(z-x_{i}\right)_{j}\right|\geq 2M_i$ we want $x_{i+1}$ to
  be closer to $z$, so set $\sign
  y_{j}=\sign\left(z-x_{i}\right)_{j}$. We will see later that this
  automatically takes care of the distance from $\partial Q_r$.
\label{enu:closer z}
\item Otherwise we ignore the distance from $z$ and just pull away
  from $\partial Q_r$ i.e.\ set
$\sign y_{j}=-\sign\left(x_{i}\right)_{j}$. For notational
  convenience, assume here and below that $\sign 0 = 1$.
\label{enu:furtherbdry}
\end{enumerate}
This concludes the description of the construction, and we
automatically get the connection probability estimate (\ref{eq:xixi+1}).

Next we wish to verify that we indeed reach a neighbor of $z$ in at most
$C\log r$ steps and that (\ref{eq:farboundary}) holds. We shall show that every step of the induction
does not increase $||z-x_i||_\infty$ and after $d$ steps the norm is
reduced by a constant i.e.
\begin{equation}\label{eq:zxi34}
||z-x_{i+d}||_\infty\le \tfrac{3}{4}||z-x_i||_\infty
\end{equation}
which is enough. We first note that by \ref{enu:closer z}, the fact that
$||y||_{\infty}\leq M_i$ and that $||z-x_i||_\infty=4M_i$ it is immediately clear that
\begin{equation}
\left\Vert z-x_{i+1}\right\Vert _{\infty}\leq
\left\Vert z-x_{i}\right\Vert_{\infty}.\label{eq:inftyle}
\end{equation}
Further, since $y\in\partial Q_{M_i}$ then it must have at least one
coordinate with absolute value $M_i$. Denote by $j_1$ the largest coordinate in absolute value of $z-x_i$.
We get that $|\left(z-x_i\right)_{j_1}|$ is reduced from
$4M_i$ to $3M_i$. Again from \ref{enu:closer z} we
see that at the next steps it will stay below $3M_i$, because it can
only increase (at some step $i+k$) if it becomes $\le 2M_{i+k}$ and in
this case it can only increase up to $3M_{i+k}$. This is $\le 3M_i$
by (\ref{eq:inftyle}). In short we get
\begin{equation}\label{eq:goodcoord}
\left|\left(z-x_{i+k}\right)_{j_1}\right|\le 3M_i\quad\forall k\ge 1.
\end{equation}
Next denote by $j_2$ the largest coordinate of $z-x_{i+1}$. If
$j_2=j_1$ then $||z-x_{i+1}||_\infty \le 3M_i$ and (\ref{eq:zxi34}) is
proved. Otherwise we get from the same arguments
\[
\left|\left(z-x_{i+1+k}\right)_{j_2}\right|\le 3M_{i+1}\le 3M_i
\quad\forall k\ge 1.
\]
And so on. By step $i+d$ we would have either covered all coordinates
or run into a case of two equal $j$-s, either which
demonstrates (\ref{eq:zxi34}) and hence that $N\le C\log r$.

To complete the induction we need to show that
(\ref{eq:farboundary}) is preserved. Clearly it holds for $i=0$. For
$i>1$ we have two cases:

\subsubsection*{The case $\left|\left(z-x_{i}\right)_{j}\right|<2M_i$}
 In this case, by \ref{enu:furtherbdry}, we try to increase the
 distance from $r$, and  we succeed unless
 $\big|\big(x_i\big)_j\big|<
\frac{1}{2}\big|\big(y_i\big)_j\big|$. If
we succeed then
\begin{equation}
r-\left|\left(x_{i+1}\right)_{j}\right|
  \geq
r-\left|\left(x_{i}\right)_{j}\right|
  \stackrel{(\ref{eq:farboundary})}{\geq}
M_i\stackrel{(\ref{eq:inftyle})}{\geq}
M_{i+1}\label{eq:distinc}
\end{equation}
where the reference to (\ref{eq:farboundary}) in the formula above
is a reference to our inductive assumption of the validity of
(\ref{eq:farboundary}) in the previous step. If we failed, then
\[
r-\left|\big(x_{i+1}\big)_j\right|\ge
r-\left|\left(y_i\right)_j\right|\ge 4M_i-M_i
\stackrel{(\ref{eq:inftyle})}{>}M_{i+1}\,.
\]

\subsubsection*{The case $\left|\left(z-x_{i}\right)_{j}\right|\ge 2M_i$}
If adding $y$ increases the distance of $x$ to $\partial Q_r$ then the
argument of (\ref{eq:distinc}) applies with no change. If not,
then we must have that
\[
\left(x_i\right)_j\cdot\sign((z)_j)\in
\left[-\tfrac{1}{2}|y_j|,|(z)_j|-2M_i\right]
\]
so
\[
\left(x_{i+1}\right)_j\cdot\sign((z)_j)\in
\left[\tfrac{1}{2}|y_j|,|(z)_j|-M_i\right]\,.
\]
But in this case
\[
r-\left|\left(x_{i+1}\right)_j\right|\ge
M_i \stackrel{(\ref{eq:inftyle})}{\geq}M_{i+1}\,.
\]
Together with (\ref{eq:distinc}) this shows that
(\ref{eq:farboundary}) is preserved inductively and hence holds
for all $i$. This shows that the induction is valid and proves the lemma.
\end{proof}
\begin{corollary}\label{twosquares} Let $r>s>0$ and let
$x\in\Z^d$ such that
$$ (x+Q_s) \cap \partial Q_r \neq \emptyset \, .$$
Then for any $y\in  B:=(x+Q_s)\cap Q_r$
\[
\prob(y\stackrel{B}{\lr}\partial Q_r)\ge e^{-c\log^2 s}\,.
\]
\end{corollary}
\begin{figure}
\input{cor32.pstex_t}
\caption{\label{cap:cor32} Corollary \ref{twosquares}}
\end{figure}
See figure \ref{cap:cor32}.
\begin{proof}$B$ is a box (i.e.\ with the
  sides parallel to the axis, but their length not necessarily equal). Denote by $\ell$ its
  shortest edge so that $\ell \le 2s$. It is now easy to see that one
  can find a cube $z+Q_{\ell/2}\subset B$ containing both $y$ and at
  least one point from $\partial Q_r$ --- just construct $z$
coordinate by coordinate, they are independent. And now write
\begin{align*}
&&\prob(y\stackrel{B}{\lrlong}\partial Q_r)&
  \ge \prob(y\stackrel{B}{\lrlong}z\mbox{ and
  }z\stackrel{B}{\lrlong}\partial Q_r)\ge\\
&\mbox{by FKG}&&
  \ge
  \prob(y\stackrel{B}{\lrlong}z)\prob(z\stackrel{B}{\lrlong}\partial Q_r)\ge\\
&&& \ge \prob(y\stackrel{z+Q_{||y-z||_\infty}}{\leftrightarrowfillCS}z)
      \prob(z\stackrel{z+Q_{\ell/2}}{\leftrightarrowfillCS}\partial Q_r)\ge\\
&\lefteqn{\mbox{by Lemma \ref{connect}}}&&
  \ge
  c\exp\left(-C\log^2||y-z||_\infty\right)\exp\left(-C\log^2(\ell/2)\right)\ge\\
&&& \ge c\exp(-C\log^2 s)
\end{align*}
as required.
\end{proof}


\section{{\bf A regularity theorem}}\label{sec:ldv}
In the following we prove a regularity result which is the key element
in the proof of Theorem \ref{lowmass} in chapter \ref{sec:lowmass}. We
recommend the reader reads first \S \ref{sec:global}, containing
the required definitions and the statement of the theorem, then read
how it is used in chapter \ref{sec:lowmass} and especially in Lemma
\ref{sume1e2e3} before returning to the proof of the regularity
theorem, which is the bulk of this chapter.

\subsection{Statement of the regularity theorem.}\label{sec:global}
We are interested in estimating the tails of random variables of the
form $|\C(x) \cap Q_s|$. For any particular $x$ this can easily be
done using (\ref{tpt}), the BK inequality and a moment calculation. In
fact, this is exactly performed in \cite{A}. Let us therefore define
the event that the cluster is ``typical'',
\be\label{eq:defSs}
\Ss_s(x)=\left\{
  \left|\mathcal{C}(x)\cap (x+Q_s)\right|<s^4\log^7s
\right\}\, .
\ee
As discussed in the introduction (see (\ref{eq:expo})),
$\prob(\Ss_s(x))>1-e^{-c\log^7 s}$. Where we deviate from the
simplified sketch in the introduction is in the following definition:
\begin{definition}\label{def:globalbad} For $x \in \partial Q_j$ and  positive integers $s$ and $K$ we define the following events.

\begin{enumerate}
\item We say that $x$ is $s$-bad if $\C(x; Q_j)$ satisfies
\be\label{globalbadness} \prob \big ( \Ss_s(x) \, \mid \, \C(x; Q_j) \big ) \leq 1-\exp(- \log^2 s) \, .\ee

\item We say that $x\in \partial Q_j$ is $K$-irregular if
there exists $s\geq K$ such that $x$ is $s$-bad. Otherwise we say that $x$ is $K$-regular.
\end{enumerate}
\end{definition}
The notation ``$\mid\,\C(x;Q_j)$'' means that we condition on all open
edges between two vertices of the cluster $\C(x;Q_j)$ as well as on
all closed edges with both vertices in $Q_j$ and at least one vertex
in $\C(x;Q_j)$. Shortly, on all information needed to calculate
$\C(x;Q_j)$ precisely. Note that we do not condition on edges leading
outside of $Q_j$.

Let us briefly discuss the significance of Definition
\ref{def:globalbad}. Typically $j$ is large and $s$ is
$j^{o(1)}$. Clearly the event that $x$ is $s$-bad is unusual,
due to the power of the log being $2$ in (\ref{globalbadness}) and $7$
in (\ref{eq:defSs}). The event that $x$ is $s$-bad depends on the
status of edges in $\C(x; Q_j)$ and indirectly reveals that the
boundary of the cluster (at $\partial Q_j$) is sufficiently spread
out. This is best illustrated by the following two examples of bad
configurations.
\begin{figure}
\input{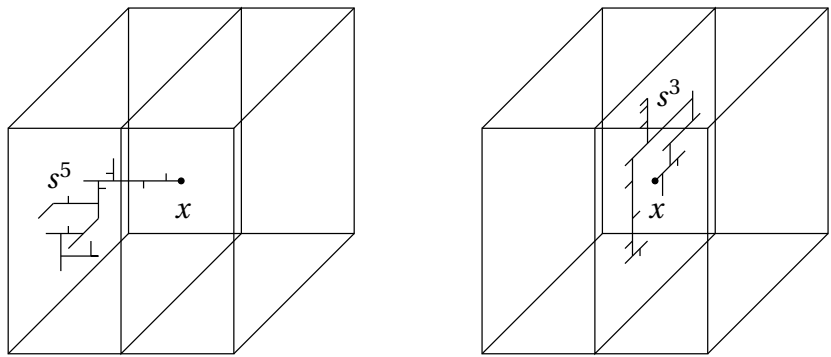}
\caption{\label{cap:2bad} Two kinds of bad configuration. On the left,
  a simple bad configuration with too many vertices inside the left
  half-cube. On the right, a bad configuration with too many vertices
  on the surface of the half-cube.}
\end{figure}

The first is a \label{pg:simplebad} ``simple'' bad configuration. See
figure \ref{cap:2bad}, left. In
this case the configuration $\C(x; Q_j)$ has a cluster of size at
least $s^4\log^7 s$ inside $(x+Q_s) \cap Q_j$ so the conditional
probability in (\ref{globalbadness}) is just $0$. The second, and more
interesting (see figure \ref{cap:2bad}, right) is when the
configuration has an excess of points on the boundary, say $s^3$ such
points. In this case, heuristically we expect
\be\label{eq:cmplbad}
\E\big(|\C(x; x+Q_s)|\;\big|\; \C(x; Q_j) \big)\approx s^5\,.
\ee
Roughly, each point on the boundary gives rise
to an expected $s^2$ points in $(x+Q_s)\setminus Q_j$, so assuming
that the part of the cluster on the boundary is sufficiently spread
out, they do not interfere negatively and you
get (\ref{eq:cmplbad}). This of course means that $x$ is bad. We shall not
justify these heuristics --- they also require some additional assumptions --- but we hope it gives the reader some intuition nonetheless. This
example shows how our definition gives information about
the behavior of the cluster $\C(x;Q_j)$ on the boundary of $Q_j$.
The alternative way, analyzing the behavior of the
cluster on the boundary explicitly, while possible, is far more
complicated.

We write $X_j^{K{\textrm{-irr}}}$ for
$$ X_j^{K{\textrm{-irr}}} = \Big | \big \{ x \in \partial Q_j : 0 \stackrel{Q_j}{\lrlong} x \and x \textrm{ is } K\textrm{-irregular} \big \} \Big | \, .$$

We are now ready to state the main theorem of this chapter.

\begin{theorem}\label{globalexploreshafan}
There exists constants $C>c>0$ such that for any $K$ sufficiently
large and any $j$ and $M$ we have
$$ \prob \Big ( X_{j} \geq M \and X_j^{K{\textrm{-irr}}} \geq  X_{j}/2
\Big ) \leq C j^d \exp(-c\log^{2}M) \, .$$
\end{theorem}
Now is the time to skip to chapter \ref{sec:lowmass}.

\subsection{Global and local regularity.} Theorem
\ref{globalexploreshafan} is the formulation needed in chapter
\ref{sec:lowmass} to prove Theorem \ref{lowmass}. It is natural to
prove such a large deviation estimate using an exploration procedure
which exploits the independence between difference boxes in the
lattice. However, the event defined in Definition \ref{def:globalbad}
is a global definition, because we need to examine the edges of the
entire cluster $\C(x; Q_j)$ in order to determine it. To that aim, we
define local events which can be determined by observing boxes of side
length polynomial in $s$.

\begin{definition} \label{def:badevents} For $x \in \partial Q_j$ and a positive integer $s$ we say that the event $\Ss_s^\textrm{loc}(x)$ occurs if the following two happen:
\begin{enumerate}
\item[(a)] For all $y \in x+Q_s$,
\[
\left|\C(y;x+Q_{s^{2d}})\cap (x+Q_s)\right|<s^4\log^4s\;\textrm{; and}
\]
\item[(b)] There exists at most $\log^3 s$ disjoint open paths
starting in $x+Q_s$ and ending at $x+\partial Q_{s^{2d}}$.
\end{enumerate}
\end{definition}
Note that in (a) we are interested in points in $x+Q_s$ but we allow
the connecting paths to traverse in a much larger set --- $x+Q_{s^{2d}}$ --- but not unlimited. We immediately note
\begin{claim} \label{disjointpaths}
For any $x\in\Z^d$ and positive integer $s$,
$$ \Ss_s^\textrm{loc}(x) \Longrightarrow \Ss_s(x) \, .$$
\end{claim}
\begin{proof}
Indeed, assume to the contrary that $$|\C(x) \cap (x+Q_s)| \geq s^4
\log^7 s \, ,$$ and let $X=\C(x)\cap (x+Q_s)$. We say that two vertices in $X$
are equivalent if there is an open path connecting them which does
not exit $x+Q_{s^{2d}}$. Due to $(a)$ from the definition of
$\Ss_s(x)$, each equivalence class contains at most $s^4\log^4s$
vertices. Due to $(b)$, there are no more than $\log^3 s$ equivalence classes, since each class requires its own path from $x+Q_s$ to the outside of $x+Q_{s^{2d}}$ and all these paths are disjoint.
\end{proof}

With this local version of $\Ss_s$, we are ready to give a local
definition of badness.


\begin{definition} \label{def:localbad} For $x \in \partial Q_j$ and positive integers $s$ and $K$ we define the following.
\begin{enumerate}

\item We say that a cluster $\C$ in $B:=(x+Q_{s^{4d^2}}) \cap Q_j$ is a
``spanning cluster'' if $\C\cap Q_j$ intersects both
$x+\partial Q_{s^{4d^2}}$ and $x+\partial Q_{s^{2d}}$. See figure
\ref{cap:spanning}. For notational convenience, we will also consider
the cluster of $x$ as spanning even if it does not span anything.
\begin{figure}
\input{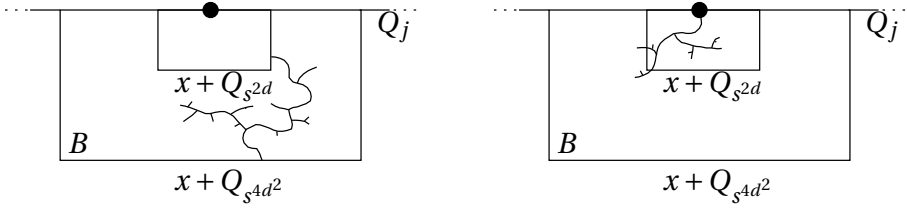}
\caption{\label{cap:spanning} Spanning clusters. On the left, a
  cluster spanning from the outer boundary to the inner. On the
  right, the cluster containing $x$. }
\end{figure}

\item We say that $x$ is $s$-locally-bad if there exists spanning clusters $\C_1, \ldots, \C_m$ in $B$ such that
\be\label{badness} \prob \Big ( \Ss_s^\textrm{loc}(x) \, \mid \, \C_1, \ldots, \C_m  \Big ) \le 1-e^{-\log^2 s} \, .\ee

\item We say that $x\in \partial Q_j$ is $K$-locally-irregular if
there exists $s \geq K$ such that $x$ is $s$-locally-bad. Otherwise we say that $x$ is $K$-locally-regular.
\end{enumerate}
\end{definition}
\noindent The importance of this definition is the fact that the event that $x$ is $s$-locally-bad is determined by the status of the edges in the box $(x+Q_{s^{4d^2}}) \cap Q_j$. Let us proceed with observing that global goodness is implied by its local counterpart. We say that $x$ is $s$-good ($s$-locally-good) if it is not $s$-bad ($s$-locally-bad).

\begin{claim} \label{localglobal} For any $x \in \partial Q_j$ and a positive integer $s$ we have that if $x$ is $s$-locally-good, then $x$ is $s$-good.
\end{claim}
\begin{proof}
Assume to the contrary that
$x$ is $s$-bad i.e. that
\[
\prob(\Ss_s(x)\,\mid\,\C(x;Q_j))\le 1-\exp(-\log^2 s)\,.
\]
By claim \ref{disjointpaths}, $\Ss_s^\textrm{loc}\Longrightarrow\Ss_s$
therefore
\[
\prob(\Ss_s^\textrm{loc}(x)\,\mid\,\C(x;Q_j))\le 1-\exp(-\log^2 s)\,.
\]
Now, the event $\Ss_s^\textrm{loc}(x)$ depends only on what happens in
$\C(x;Q_j)\cap(x+Q_{s^{2d}})$ so
\[
\prob(\Ss_s^\textrm{loc}(x)\,\mid\,\C(x;Q_j)) =
\prob(\Ss_s^\textrm{loc}(x)\,\mid\,\C(x;Q_j)\cap(x+Q_{s^{2d}}))\,.
\]
Examine now the cluster in the bigger box $x+Q_{s^{4d^2}}$ and write
it as a union of its components,
\[
\C(x;Q_j)\cap(x+Q_{s^{4d^2}})=\C_1\cup\C_2\cup\dotsc
\]
By definition, only the spanning clusters intersect the smaller box
$x+Q_{s^{2d}}$. Assume the spanning clusters are
$\C_1,\dotsc,\C_m$. We get
\[
\C(x;Q_j)\cap(x+Q_{s^{2d}})=\bigcup_{i=1}^m \C_i\cap(x+Q_{s^{2d}})\,.
\]
so
\begin{align*}
\prob(\Ss_s^\textrm{loc}(x)\,\mid\,\C(x;Q_j)) &=
\prob(\Ss_s^\textrm{loc}(x)\,\mid\,\C_1\cap(x+Q_{s^{2d}}),\dotsc,\C_m\cap(x+Q_{s^{2d}}))\,.
\intertext{and again by locality this equals}
&\prob(\Ss_s^\textrm{loc}(x)\,\mid\,\C_1,\dotsc,\C_m)\,.
\end{align*}
Hence we get that $x$ is $s$-locally-bad (with these
$\C_1,\dotsc,\C_m$), in contradiction.
\end{proof}

\noindent We write $X_j^{K{\textrm{-loc-irr}}}$ for
$$ X_j^{K{\textrm{-loc-irr}}} = \Big | \big \{ x \in \partial Q_j : 0 \stackrel{Q_j}{\lrlong} x \and x \textrm{ is } K\textrm{-locally-irregular} \big \} \Big | \, .$$


\noindent We will spend the rest of this chapter in proving the following theorem.

\begin{theorem}\label{localexploreshafan}
There exists constants $C>c>0$ such that for any $K$ sufficiently
large and any $j$ and $M$ we have
$$ \prob \Big ( X_{j} \geq M \and X_j^{K{\textrm{-loc-irr}}} \geq  X_{j}/2
\Big ) \leq C j^d \exp(-c\log^{2}M) \, .$$
\end{theorem} \medskip

\begin{proof}[Proof of Theorem \ref{globalexploreshafan}.] This
follows directly from Theorem \ref{localexploreshafan} and Claim
\ref{localglobal}.
\end{proof}



\subsection{An easy large deviation estimate.}
Our aim in this section is to prove the following lemma, which will be crucial for the proof of Theorem \ref{localexploreshafan}.
\begin{lemma} \label{easyshafan} For $x \in \partial Q_j$ and positive integers $s$ we have
\[
\prob(x\mbox{ is }s\mbox{-locally-bad}) \le Ce^{-c\log^4 s}\,.
\]
\end{lemma} \medskip

\noindent In order to prove this lemma, we begin with a large deviation lemma.
\begin{lemma} \label{volld}
There exists some constant $c>0$ such that for all $s>0$ and
$\lambda>0$ we have
$$
\prob ( \max_{y\in Q_s}|\C(y)\cap Q_s| > \lambda s^4 ) \leq
  s^{d-6}  e^{-c \lambda} \, .
$$
\end{lemma}
\begin{proof} Denote this maximum by $\C_{\max}$. Our lemma is a
well-known corollary of the so-called ``diagrammatic bounds''. A
convenient reference is \cite[\S 4.3, lemma 2, eq.\ (4.12)]{A}. It
states that
\[
\E(\C_{\max}^k)\le k! C_1^k s^{d-6+4k}\, ,
\]
(Aizenman's $\eta$ is simply $0$ in our case). Using this with
$k=\lambda/2C_1$ gives
\begin{align*}
\prob(\C_{\max}>\lambda s^4)&
  =\prob\left(\C_{\max}^k>\left(\lambda s^4\right)^k\right)
  \le \frac{\E(\C_{\max}^k)}{\left(\lambda s^4\right)^k}
  \le k!s^{d-6}\left(\frac{C_1}{\lambda}\right)^k\\
&\le s^{d-6}\left(\frac{C_1k}{\lambda}\right)^k
  = s^{d-6} 2^{-\lambda/2C_1}.\qedhere
\end{align*}
\end{proof}


\begin{proof}[Proof of lemma \ref{easyshafan}] Indeed, by Lemma
\ref{volld}, the probability of $(a)$ from definition \ref{def:badevents} of
$\Ss_s(x)$ is at most $Ce^{-c\log^4 s}$. As for (b), by the volume estimate (\ref{deltais2}) we see that
\begin{eqnarray*}
\lefteqn{\prob\big(x+Q_s\lrlong x+\partial Q_{s^{2d}}\big)
  \le} & & \\
&\qquad &\le \sum_{y\in x+\partial Q_s}\prob(y\lrlong x+\partial
Q_{s^{2d}})\le\\
&& \le \sum_{y\in x+\partial Q_s}\prob(|\C(y)|>s^{2d}-s)\le\\
\mbox{by (\ref{deltais2})}&& \le \sum_{y\in x+\partial Q_s}
  C\left(s^{2d}-s\right)^{-1/2} \le \\
&&\le Cs^{d-1}\cdot Cs^{-d}=C/s\,.
\end{eqnarray*}
We now apply the BK inequality and we get that the probability of (b) is at most $(C/s)^{\log^3 s}\leq Ce^{-c\log^4 s}$. We deduce that
\be\label{badcube.deviate} \prob \big ( \Ss_s(x) \big ) \ge 1- Ce^{-c\log^4 s} \, .\ee
Similarly to above, for any $y \in x+\partial Q_{s^{2d}}$ by (\ref{deltais2}) we have that
$$ \prob (y \lr x+\partial Q_{s^{4d^2}}) \leq Cs^{-2d^2} \, ,$$
whence the probability that there exists more than $\log^3 s$ spanning clusters (as in Definition \ref{def:localbad}) is at most
$$ \left(Cs^{2d^2-1}\right)^{\log ^3 s} \left(Cs^{-2d^2}\right)^{\log ^3 s} \leq C e^{-c\log^4 s} \, ,$$ since we have $Cs^{2d^2-1}$ possible choices of $y$.

If $x$ is $s$-locally-bad, then either there are at least $\log^3 s$
spanning clusters, or there are at most $\log^3 s$ spanning clusters
and a certain subset of them $\C_1, \ldots, \C_m$ (we assume here these spanning clusters are numbered in some arbitrary fashion) has the property that
\be\label{clustersubset} \prob \big ( \Ss_s(x) \, \mid \, \C_1, \ldots
\C_m  \big ) \le 1- e^{-\log^2 s} \, .\ee
However, for each such subset, by (\ref{badcube.deviate}) we have that
$$ \E \,\, \prob \big ( \neg\Ss_s(x) \, \mid \,
\C_1, \ldots \C_m  \big )  \le Ce^{-c\log^4 s} \, .$$
Thus, by Markov inequality, the probability that $\C_1, \ldots, \C_m$
have the property (\ref{clustersubset}) is at most $C\exp(\log^2
  s-c\log^4 s)\le C\exp({-c\log^4 s})$. We conclude the proof using
  the union bound, since $2^{\log ^3 s}\cdot Ce^{-c\log^4 s} \leq Ce^{-c\log^4 s}$.
\end{proof}




\newcommand{\perhapswillbeback}
{
\noindent {\bf Proof of part $(1)$ of Theorem \ref{shafan}.}
\label{page:prf113}
Denote $Q=Q_{s^{2d+1}}$. We first assume that
$w \not \in x+Q$ and that $s^{2d+1}<r$ (this is the more interesting case). We
now use Proposition \ref{disjointpaths}
and get an inclusion of events
\[
\{Y_{s,r}>s^4\log^5 s\}\subset  (a)\cup (b)
\]
where the events $(a)$ and $(b)$ are correspond to the two clauses in Proposition
\ref{disjointpaths}. The advantage of this presentation is that both
$(a)$ and $(b)$ depend only on the configuration inside
$x+Q$ and hence we may write for any configuration $\xi$ on the
outside of $x+Q$,
\[
\prob(Y_{s,r}>s^4\log^5 s\,|\,\xi)\le \prob((a))+\prob((b))\,.
\]
Now, $(a)$ is estimated directly by Lemma
\ref{volld} and we get
\[
\prob((a))\le s^{d-6} e^{-c\log^3 s}\leq 2e^{-c'\log^3 s}\,.
\]
For $(b)$ we use Barsky \& Aizenman (\ref{deltais2}) to get, for any
$y\in Q_{s}$ that
\[
\prob(y\lr \partial Q)
  \le \prob(|\C_y|>s^{2d+1}-s)
  \stackrel{(\ref{deltais2})}{\le}
    C\left(s^{2d+1}-s\right)^{-1/2}
\]
and summing over all $y$ gives
\[
\prob(Q_s\lr\partial Q)\le Cs^{-1/2}\,.
\]
Since the event $(b)$ is simply that there exist $\log^2 s$ such
connections and they are all disjoint, the BK inequality gives
\[
\prob((b))\le \left(Cs^{-1/2}\right)^{\log^2 s}\le 2e^{-c\log^3 s}\,.
\]
Summing $(a)$ and $(b)$ gives
\[
\prob(Y_{s,r}>s^4\log^5 s\,|\,\xi)\le 2e^{-c\log^3 s}\,,
\]
our first milestone.

Denote $B=(x+Q)\cap Q_r$. We now apply Corollary
\ref{twosquares}. We recall that it claims that for any $y\in\partial B$,
\[
\prob(y\stackrel{B}{\lrlong}x)\ge e^{-c\log^2 (s^{2d+1})}=e^{-c'\log^2 s}\,.
\]
This might not be large, but it dwarfs $2e^{-c\log^3 s}$ by orders of
magnitude, so we can write
\begin{equation}\label{eq:Ysrxy}
\prob(Y_{s,r}>s^4\log^5 s\,|\,\xi)\le
  2e^{-c\log^3 s}\prob(y\stackrel{B}{\lrlong}x)
\end{equation}
where the $e^{-c\log^2 s}$ was simply swallowed into the constant.

With this established we can finish this case. Call a
configuration $\xi$ admissible if
$w\stackrel{Q_r}{\lrlong}\partial Q$. Clearly if
$w\stackrel{Q_r}{\lrlong}x$ then the configuration outside $x+Q$
must be admissible. Further, let $y(\xi)$ be some
point on $\partial Q$ connected to $0$ in $Q_r$ --- if more than one
exists, choose one arbitrarily. With these notations we can write
\begin{align*}
&&\lefteqn{\prob(w\stackrel{Q_r}{\lrlong}x
  \mbox{ and }Y_{s,r}>s^4\log^5 s) = }\qquad&\\
&&&= \sum_{\xi\;\textrm{admissible}}\prob(\xi)\prob(Y_{s,r}>s^4\log^5
s\mbox{ and }0\stackrel{Q_r}{\lrlong}x \,|\,\xi)\\
&&&\le \sum_{\xi\;\textrm{admissible}}\prob(\xi)\prob(Y_{s,r}>s^4\log^5
s \,|\,\xi)\\
&\mbox{by (\ref{eq:Ysrxy})}&&\le
  \sum_{\xi\;\textrm{admissible}}\prob(\xi)\cdot
  2e^{-c\log^3 s}\prob(y(\xi)\stackrel{B}{\lrlong}x)\\
&&&\le 2e^{-c\log^3 s}\prob(w\stackrel{Q_r}{\lrlong}x)
\end{align*}
and we are done with the case that $w\notin x+Q$ and $s^{2d+1}<r$.

We now assume that $w\in x+Q_{2s^{2d+1}}$ which covers both $w\in x+Q$
and $s^{2d+1}\ge r$. We use Lemma \ref{volld} to bound
\[
\prob \Big ( w \stackrel{Q_r}{\lrlong} x \and
             Y_s(x) \geq s^4\log^{5}s  \Big )
\leq \prob \Big ( Y_s(x) \geq s^4\log^{5}s  \Big )
\leq e^{-c \log^5 s} \, .
\]
On the other hand, define $B=(x+Q_{2s^{2d+1}})\cap Q_r$ and use
Corollary \ref{twosquares} to get
\[
\prob \Big ( w \stackrel{Q_r}{\lrlong} x \Big ) \geq
\prob\Big ( w \stackrel{B}{\lrlong} x\Big)
  \ge e^{-c\log^2 (2s^{2d+1})}
  \ge e^{-c'\log^2s} \, .
\]
This concludes the proof of part (1) of the theorem.  \qed
}


\subsection{{Exploring the cluster of the origin}} \label{sec:exploreshafan}

%
In this subsection we prove Theorem \ref{localexploreshafan}. To that aim, we ``explore'' the cluster of the origin in $Q_j$ in boxes of
size $s^{4d^2}$ --- in fact we are only interested in the boundary
$\partial Q_j$. This is quite a standard procedure, but let us
describe it in detail. Let $w\in\Z^d$ be some shift, and let $G=G(w)$
be the collection of all cubes of size $2s^{4d^2}$, aligned to the
shifted grid $w+\Z^d$ and intersecting $Q_j$ i.e.\[
G:=\{(Q_{2s^{4d^2}}+v)\cap
Q_j:v\in(4s^{4d^2}+1)\mathbb{Z}^{d}+w\}\setminus\{\emptyset\} \, ,
\] and choose an arbitrary ordering of $G$. The role of the shift $w$ is rather technical and will become evident later. The exploration
process is a sequence of two subsets of $G$, $E_{i}$ (the {\em
explored} boxes) and $A_{i}$ (the {\em active} boxes). We start
with \begin{align*}
E_{1} & =\{q\in G:q\cap\partial Q_j=\emptyset\}\\
A_{1} & =\{q\in G:\exists x\in\partial q,0\stackrel{\cup
E_{1}}{\longleftrightarrow}x\}\setminus E_{1} \, ,\end{align*}
where we use the notation $\cup E_i = \cup _{q \in E_i} q$.

At step $i$ we choose from $A_{i-1}$ a box according to our ordering of
$G$. Denote it by $q_{i}$. We add $q_i$ to the set of explored
boxes $E_{i-1}$, and then add to $A_i$ all boxes not yet explored
(that is, boxes not belonging to $E_{i-1}$) which can be reached
from $0$ by paths going only through the explored boxes $E_{i-1}\cup q_i$.
Namely, \begin{align*}
E_{i} & =E_{i-1}\cup\{q_{i}\}\\
A_{i} & =\left(A_{i-1}\cup\{q \in G:\exists x\in\partial
q,0\stackrel{\cup E_{i}}{\longleftrightarrow}x\}\right)\setminus
E_{i}.\end{align*} Since the set of boxes $G$ is finite and $E_i$
increases at each step, at some
time we must have $A_{i}=\emptyset$ at which time we cannot choose
$q_{i+1}$. We say that the exploration finished, and denote this
stopping time by $\tau$.

The exploration process is used in order to define two martingales.
One to control the number bad boxes ($\beta_{i}$) and one to control the
boundary vertices ($\gamma_{i}$). We say that a box $q\in G$ is
$s$-bad if there exists some $x\in\partial Q_j$ which is $s$-locally-bad and such that $(x+Q_{s^{4d^2}})\cap Q_j\subset q$. Both martingales are adaptable
to the exploration filtration $\{\F_i\}$, namely to the configuration restricted
to $\cup_{j\leq i} E_{j}$. Their definition is as follows. We start
with $\beta_{1}=\gamma_{1}=0$. At each step we define
\begin{align*}
\beta_{i} & =\beta_{i-1}+\boldsymbol{1}\{q_{i}\mbox{ is $s$-bad}\}-\prob(q_{i}\mbox{ is $s$-bad} \, \mid \, \F_{i-1})\nonumber \, ,\\
 \gamma_{i} & =\gamma_{i-1}+\boldsymbol{1}\{\exists x\in
q_{i}\cap\partial Q_j:0\stackrel{\cup
E_{i}}{\longleftrightarrow}x\}\;-\\
&\qquad\qquad\qquad-\prob(\exists x\in
q_{i}\cap\partial Q_j:0\stackrel{\cup
E_{i}}{\longleftrightarrow}x\, \mid \, \F_{i-1}) \, ,
\label{eq:defgam}\end{align*}
We
extend $\beta_i$ and $\gamma_i$ for all $i$ by making
$\beta_{i}=\beta_{i-1}$ and $\gamma_{i}=\gamma_{i-1}$ when
$A_{i-1}=\emptyset$. Clearly, they are indeed martingales. For our next lemmas, recall the definition of $X_j$ at (\ref{eq:defXj}).

\begin{lemma} \label{boundaryverts} There exists constants $C_1>0$ and
  $c_1>0$ such that for any $j,s$ and $M$ we have
$$ \prob \big ( c_{1}e^{-C_{1}\log^{2}s}\tau \ge X_{j} \ge M\big )
  \leq Ce^{-cM+C\log^2s} \, .$$
where $\tau$ is the stopping time for the exploration defined above.
\end{lemma}
When we apply the lemma $s\ll M$ so you may think about the right-hand
side as $Ce^{-cM}$.
\begin{proof}
For every $i$,\begin{align*}
X_{j} & \ge \big |\{k\leq i:\exists x\in q_{k}\cap\partial Q_j:0\stackrel{\cup E_{k}}{\longleftrightarrow}x\} \big |=\\
 & =\gamma_{i}+\sum_{k=1}^{i}\prob(\exists x\in q_{k}\cap\partial Q_j:0\stackrel{\cup E_{k}}{\longleftrightarrow}x \, \mid \, \F_{k-1}).\end{align*}
To bound from below the sum on the right-hand, note the fact that we explored
$q_{k}$ implies that there exists some $z\in\partial q_k$
which is connected to $0$ in $\cup E_{k-1}$. This means that given
$\F_{k-1}$, the probability that there exists $x\in
q_{k}\cap\partial Q_j$ such that $0\stackrel{\cup
E_{k}}{\longleftrightarrow}x$ is at least the probability that
$z\stackrel{q_{k}}{\leftrightarrow}\partial Q_j$ which by Corollary
\ref{twosquares} is at least
$ce^{-C\log^{2}(2s^{4d^2})}\ge c_2e^{-C_2\log^2 s}$ for some $C_2>0$
and $c_2>0$. Assume also that $c_2<\frac{1}{2}$. This holds whenever the
exploration at time $k$ is still alive. We will prove the assertion of the lemma with constants $c_1=c_2^2$, and $C_1=2C_2$. For brevity write $\mu=c_1\exp(-C_1\log^2 s)$ and so we get that for any
$i>0$
\[
X_{j}\ge\gamma_{i}+\mu^{1/2}\min\{i,\tau\}\,.
\]
We rearrange to get that
\begin{eqnarray*}
\lefteqn{\prob(\mu\tau \ge X_{j}\ge M)\le} & \lefteqn{}\\
 & \qquad & \le\sum_{i=M/\mu}^{\infty}
  \prob(\tau=i\mbox{ and }X_{j}\le \mu i)\le\\
 &  &
\le\sum_{i= M/\mu}^{\infty}
  \prob(\gamma_{i}\le -(\mu^{1/2}-\mu) i)
\end{eqnarray*}
We assumed $c_2<\frac{1}{2}$ which gives $\mu<\frac{1}{4}$ and hence $\mu^{1/2}-\mu \geq {\mu^{1/2} / 2}$. We now use the Azuma-Hoeffding inequality, which asserts that for
any $a$, $\prob(\gamma_{i}\le -a)\leq e^{-a^{2}/2i}$. We get that
\begin{equation*} \prob\big (\mu\tau \geq
X_{j}\geq M\big )
\le\sum_{i=M/\mu}^{\infty}\exp(-\mu i/4)\le
\frac{C}{\mu}e^{-cM} \, ,\label{eq:endsmall}\end{equation*} concluding the
proof of the lemma.
\end{proof}

The counterweight to Lemma \ref{boundaryverts} is the following lemma
which estimates the bad vertices. Recall that $G$ and hence $\beta_i$,
$\gamma_i$ and $\tau$ all depended on a parameter $w$, the
shift. Denote now
$$ X_j^{s\textrm{-bad}} = \left|\left\{x\in\partial Q_j:
  x\mbox{ is } s\mbox{-locally-bad}\right\}\right|\, ,$$
  and
\[
X_j^\mathrm{bad}(w)= \left|\left\{x\in\partial Q_j:
  x\mbox{ is } s\mbox{-locally-bad, and }
  \exists q\in G(w)\mbox{ s.t.\ }x\in q
\right\}\right|\,.
\]
Of course, to conclude the proof of Theorem \ref{localexploreshafan} we
will need to use the fact that $X_j^{s\textrm{-bad}}$ is bounded by a sum of $X_j^{\mathrm{bad}}(w)$ for $2^d$ different $w$-s. But for now let us examine one $w$ only.
\begin{lemma} \label{badverts} There exist constants $C_3>0$ and
  $c_3>0$ such that for any $j$, $s$, $w$ and $M$;
and any real number $\mu \geq C_{3}e^{-c_{3}\log^{4}s}$ we have
$$ \prob \big (\mu^{-1} X_j^{\textrm{bad}}(w) \ge \tau \ge M \big )
\leq \frac{C}{s^{2-2d}\mu^2}\exp(-cs^{2-2d}\mu^{2} M) \, .$$
\end{lemma}
Again, the way we apply this lemma most of the factors on the
right-hand side are negligible, and one can consider it as
$\exp(-\mu^2 M)$.
\begin{proof}
We have \[ \beta_{i}=|\{k\leq i:q_{k}\mbox{ is
bad}\}|-\sum_{k=1}^{i}\prob(q_{k}\mbox{ is bad} \, \mid \,
\F_{k-1}).\] By Lemma \ref{easyshafan} we have that
\begin{eqnarray*}
\lefteqn{\prob(q_{k}\mbox{ is bad} \, \mid \, \F_{k-1})\leq}&&\\
&\qquad&\le\sum_{x:x+Q_{s^{4d^2}}\subset q_k}
  \prob(x\mbox{ is } s\mbox{-locally-bad}\,\mid\,\F_{k-1})=\\
\mbox{by locality}&&=\sum_{x:x+Q_{s^{4d^2}}\subset q_k}
  \prob(x\mbox{ is } s\mbox{-locally-bad})\le\\
\mbox{by Lemma \ref{easyshafan}} && \le \sum_{x:x+Q_{s^{4d^2}}\subset q_k}
  Ce^{-c\log^4 s}\le \\
&&\le Cs^{4d^3}\cdot Ce^{-c\log^4 s} \le
Ce^{-c\log^{4}s}
\end{eqnarray*} and so
\[
\beta_{i}\ge|\{k\leq i:q_{k}\mbox{ is bad}\}|-Ce^{-c\log^{4}s}i\,.
\] Thus, it holds deterministically that
\[ X_j^{\textrm{bad}}(w) \le s^{d-1}|\{k\le\tau:q_{k}\mbox{ is bad}\}|\le
s^{d-1}\beta_{\tau}+C_{4} e^{-c_{4}\log^{4}s}\tau\,.
\] Define $c_3:=c_4$ and $C_3:=2C_4$. We get that $
X_j^{\textrm{bad}}(w) \ge \mu \tau$ implies that
\[
\beta_{\tau}\ge s^{1-d}\tau(\mu-\tfrac{1}{2}C_{3} e^{-c_{3}\log^{4}s})\ge
\tfrac{1}{2}s^{1-d}\mu \tau \, ,
\]
by our assumption on $\mu$. This gives
\begin{align*}
\prob(\mu^{-1} X_j^{\textrm{bad}}(w) & \ge \tau \ge M)\le
  \prob\left(\frac{2s^{d-1}}{\mu}\beta_{\tau} \ge \tau \ge M\right)= \\
& =\sum_{i=M}^{\infty}\prob(\tau=i\mbox{ and }
    \beta_{i} \ge\tfrac{1}{2}s^{1-d}\mu i)\le\\
&\le\sum_{i=M}^{\infty}\prob(\beta_{i} \ge \tfrac{1}{2}s^{1-d}\mu i)\le\\
& \le \sum_{i=M}^{\infty}\exp\left(-cs^{2-2d}\mu^{2}i\right)\le
    \frac{C}{s^{2-2d}\mu^2}\exp(-c\mu^{2}Ms^{2-2d})
\end{align*}
where for the penultimate inequality we again used Azuma-Hoeffding.
\end{proof}

\noindent {\bf Proof of Theorem \ref{localexploreshafan}.}
By definition
\[
X_j^{K\textrm{-irr}}=\bigcup_{s\geq K}^\infty
X_j^{s\textrm{-bad}}
\]
so the theorem will be proved once we get a good estimate of
\[
\prob(X_j\ge M\mbox{ and }X_j^{s\textrm{-bad}} \geq X_j/s^2)
\]
by taking a union bound over all $s$. Next we relate
$X_j^{s\textrm{-bad}}$ to $X_j^{\textrm{bad}}(w)$. Indeed, let
$W=\{w:w_i\in\{0,2s^{4d^2}\}\,\,\,\forall i=1,\dotsc,d\}$ so $W$ is a set of
$2^d$ shifts. It is easy to convince oneself that for any $x\in\Z^d$
there exists some $w\in W$ such that $x+Q_{s^{4d^2}}\subset q$ where $q$
is a cube of form $v+Q_{2s^{4d^2}}$, $v\in(4s^{4d^2}+1)\Z^d+w$. Thus
\[
X_j^{s\textrm{-bad}}=\bigcup_{w\in W}X_j^{\textrm{bad}}(w)
\]
so it is enough to estimate
$\prob(X_j\ge M\mbox{ and }X_j^\textrm{bad}(w)>X_j/2^ds^2)$. For this we write
\begin{multline*}
\prob(X_{j} \ge M\mbox{ and } X_j^{\textrm{bad}}(w) \ge X_{j}/2^ds^2)\le\prob\Big (c_{1}e^{-C_{1}\log^{2}s}\tau \ge X_{j} \ge M \Big )+\\
\prob\left(2^ds^2 X_j^\textrm{bad}(w) \ge X_{j} \ge
  \max\{M,c_{1}e^{-C_{1}\log^{2}s}\tau\}\right)
\end{multline*}
where $c_1$ and $C_1$ are from Lemma \ref{boundaryverts}.
The first term on the right is estimated by Lemma
\ref{boundaryverts} to be at most $C\exp(-cM+C\log^2s)$. The event of the
second term implies that $$ X_j^\textrm{bad} \ge 2^{-d}s^{-2}
c_{1}e^{-C_{1}\log^{2}s}\tau \, ,$$ and that $\tau \ge Ms^{1-d}$, since  $\tau \ge s^{1-d}X_{j}$. We wish to use Lemma \ref{badverts} with
$\mu= 2^{-d}s^{-2}c_{1}e^{-C_{1}\log^{2}s}$ and $M_\textrm{lemma
  \ref{badverts}}=Ms^{1-d}$. The only condition of Lemma
\ref{badverts} is that $\mu\ge C_{3}e^{-c_{3}\log^{4}s}$, which will
hold as long as $s$ is sufficiently large. We fix $K$ sufficiently large
so that any $s\geq K$ will satisfy the requirement on $\mu$ and get by Lemma \ref{badverts} that
\begin{align*}
\prob\left(2^ds^2X_j^\textrm{bad} \ge X_{j} \ge
    \max\{M,c_{1}e^{-2C_{1}\log^{2}s}\tau\}\right)
&\le \frac{C}{\mu^2s^{2-2d}}\exp\left(-c\mu^{2}Ms^{3-3d}\right) \\
&\le  C\exp\left(-ce^{-C\log^{2}s}M+C\log^2 s\right)\,.
\end{align*}
This is the larger term, so we get
\[
\prob(X_j\ge M\mbox{ and }X_j^\textrm{bad}(w)>X_j/2^ds^2)
\le   C\exp\left(-ce^{-C\log^{2}s}M+C\log^2 s\right)\,.
\]
We are nearly done. Let $s_0$ be the maximal $s$ such
that $ce^{-C\log^2 s}\ge M^{-1/2}$, so $\log s_0\approx \log^{1/2} M$. We
have
\begin{multline}\label{eq:s<s0}
\sum_{s=K}^{s_0} \prob(X_j\ge M\mbox{ and
}X_j^\textrm{bad}(w)>X_j/2^ds^2)
\le\sum_{s=K}^{s_0} C\exp\left(-ce^{-C\log^2 s}M+C\log^2s\right)\le\\\le
C\exp(-c\sqrt{M}+C\log M)\le Ce^{-c\sqrt{M}}.
\end{multline}
Summing over all $w\in W$ we get our estimate for small $s$,
\[
\sum_{s=K}^{s_0} \prob(X_j\ge M\mbox{ and
}X_j^{s\textrm{-bad}}>X_j/s^2)\le Ce^{-c\sqrt{M}}.
\]
For $s>s_0$ we use a much simpler estimate directly using Lemma
\ref{easyshafan} with no need to go through the ``martingale lemmas''
\ref{boundaryverts} and \ref{badverts} and no need for $w$. We write
\begin{align*}
\prob(X_j^{s\textrm{-bad}}>M/s^2)&\le \prob(X_j^{s\textrm{-bad}}\ge
1)\le\\
&\le\sum_{x\in Q_j}\prob(x\mbox{ is } s\mbox{-loc-bad})\le
Cj^d\exp(-c\log^4s).
\end{align*}
so
\begin{multline}\label{eq:s>s0}
\sum_{s= s_0}^\infty \prob(X_j\ge M\mbox{ and
}X_j^{s\textrm{-bad}}>X_j/s^2)\le\\
\le\sum_{s= s_0}^\infty Cj^d\exp\left(-c\log^4 s\right)\le
  Cj^d\exp\left(-c\log^{2}M\right)\,.
\end{multline}
Clearly $\mbox{(\ref{eq:s>s0})}$ is asymptotically larger than $\mbox{(\ref{eq:s<s0})}$ so
all-in-all we get
\[
\sum_{s=K}^\infty \prob(X_j\ge M\mbox{ and
}X_j^{s\textrm{-bad}}>X_j/s^2)\le Cj^d\exp(-c\log^{2}M) \, ,\]
which concludes the proof.\qed

\section{{\bf Proof of Theorem \ref{lowmass}}}\label{sec:lowmass}

\begin{wrapfigure}[5]{o}{0.33\columnwidth}%
\begin{centering}
\input{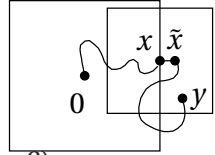}
\par
\end{centering}
\setlength{\captionindent}{0pt}
\caption{\label{cap:admis} An admissible pair $(x,y)$.}
\end{wrapfigure}
Let $j$ and $L$ be as in Theorem \ref{lowmass} and let $K$ be some
parameter sufficiently large --- we will need it to be sufficiently
large to allow to apply Theorem \ref{globalexploreshafan}, but this is
not the only restriction. We say a pair of vertices $(x,y)$ are
$(j,L,K)$-admissible if the following conditions hold
(see figure
\ref{cap:admis})

\begin{itemize}
\item $x\in \partial Q_j$ and $y\in x+Q_L$;

\item $0 \stackrel{Q_j}{\lrlong} x$ and $x \lr y$;

\item $x$ is $K$-regular; and

\item The edge $(x,\tilde{x})$ is pivotal for the event $0 \lr y$
  where $\tilde{x}$ is the neighbor of $x$ not in $Q_j$ (if more than
  one exists, choose the first in lexicographical order).
\end{itemize}
Define the random variable
$$ Y(j, K, L) = \left | \big \{ (x,y) : (x,y) \textrm{ are } (j,L,K)\textrm{-admissible} \big \} \right | \, .$$
We write $X_j^{K{\textrm{-reg}}}$ for the random variable counting the
number of $K$-regular vertices, that is, $X_j^{K{\textrm{-reg}}} = X_j
-X_j^{K{\textrm{-irr}}}$ (see the definition of
$X_j^{K{\textrm{-irr}}}$ before the statement of Theorem
\ref{globalexploreshafan}). Throughout this section, $j$, $L$ and $K$
will be fixed, and we will usually omit them from the notation, namely
we will write $Y=Y(j,K,L)$,
$X^{\textrm{reg}}=X_j^{K{\textrm{-reg}}}$ etc. The following lemmas
are the key steps in proving Theorem \ref{lowmass}.

\begin{lemma}\label{firstmoment} Let $K$ be sufficiently large, and
let $j,M$ and $L$ be integers such that $M \geq L^2/2$. Then there exists a constant $c=c(K)>0$ such that
$$ \E Y(j, K, L) {\bf 1}_{\{ X_j^{K\textrm{-reg}} = M \}} \geq cML^2
  \prob(X_j^{K\textrm{-reg}} = M) \, .$$
\end{lemma}

\begin{lemma}\label{secondmoment} Let $j$, $K$, $M$ and $L$ be integers. Then
$$ \E Y^2(j, K, L) {\bf 1}_{\{ X_j^{K\textrm{-reg}} = M \}} \leq CM^2 L^4 \prob(X_j^{K\textrm{-reg}} = M) \, .\\$$
\end{lemma}

We begin with proving Theorem \ref{lowmass} given the lemmas.
\medskip

\noindent{\bf Proof of Theorem \ref{lowmass}.}
Recall the definitions of $X_j$ and $A_j$ preceding the statement of Theorem \ref{lowmass} and denote $X=X_j$, $A=A_j$ etc\@. We begin with
\begin{align} \label{split}
\prob \big ( X \geq L^2 \and A \leq cL^4 \big ) &\leq
  \prob \big ( X \geq L^2 \and X^{\textrm{irr}} \geq L^2/2 \big)+\mbox{} \nonumber \\
&+ \sum_{M \geq L^2/2} \prob \big ( X^{\textrm{reg}} = M \and
A \leq cL^4 \big ) \, .
\end{align}
We will bound the first term using Theorem \ref{globalexploreshafan},
and each summand on the right hand side we bound using a second moment
argument with Lemmas \ref{firstmoment} and \ref{secondmoment}. We
first note that
\be\label{dominate}
A \geq Y \, .
\ee
Indeed, for each pair $(x,y)$
counted in $Y$ we have that $0 \lr y$ holds. Furthermore, we
required that for each pair $(x,y)$ counted in $Y$ the edge
$(x,\tilde{x})$ is pivotal for $0 \lr y$. This shows that $x$ must be
unique --- if both $x_1$ and $x_2$ satisfy this then
by the ``chain of sausages'' picture \cite[p. 91]{G},
one of them (say $x_2$) must be in the cluster
connected to zero only by the pivotal edge $(x_1,\tilde{x}_1)$ which
contradicts the requirement that $0 \stackrel{Q_j}{\lrlong} x_2$. This
shows (\ref{dominate}).

Recall the inequality (see \cite{D})
$$ \prob \big ( V > a ) \geq \frac{ (\E V - a)^2}{ \E V^2} \, ,$$
valid for any random variable $V \geq 0$ and $a < \E V$. We use this
for the variable $Y$ conditioned on $X^\textrm{reg}=M$ and for $a=cML^2$. Lemmas
\ref{firstmoment} and \ref{secondmoment} give that
\[
\prob(Y>cML^2\,|\,X^\textrm{reg}=M)>c\,,
\]
and the fact that $M\geq
L^2/2$ gives that there exists
positive constants $c_1, c_2$, depending on $K$, such
that
$$ \prob \Big ( Y \geq c_1 L^4 \, \mid \, X^{\textrm{reg}} = M \Big ) \geq c_2 \, .$$
We use this and the fact that $A\ge Y$ (\ref{dominate}) to derive that
$$ \prob \big ( X^{\textrm{reg}} = M \and A
\leq cL^4 \big ) \leq (1-c_2) \prob ( X^{\textrm{reg}} = M ) \,
.$$ Putting this back into (\ref{split}) and using Theorem \ref{globalexploreshafan} gives that
\begin{align*} \prob \big ( X \geq L^2 \and A \leq cL^4 \big ) &\le Cj^d e^{-c \log^2 L} + (1-c_2) \prob( X^{\textrm{reg}} \geq L^2/2)\\
 &\le Ce^{-c\log^2 j} + (1-c_2)\prob(0 \lr \partial Q_j) \, ,
\end{align*}
where we used the fact that $L\geq j^{1/10}$. The first term is
negligible (recall Lemma \ref{lower} and the fact that our theorem is
only supposed to hold for $j$ sufficiently large) and this concludes our proof.\qed \\

We proceed with the proofs of Lemmas \ref{firstmoment} and
\ref{secondmoment}. To this aim we define the following events. In
these definitions we always have $x \in \partial Q_j$ and $y \in
x+Q_L$ and $x'$ in the box $(x+Q_K)\setminus Q_{j+K/2}$, see Figure
\ref{cap:e1e2}.
\begin{figure}
\input{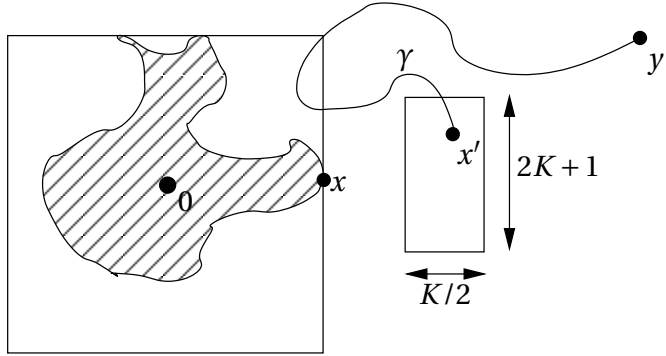}
\caption{\label{cap:e1e2} The event $\EE_1(x, M, K) \cap \EE_2(x, x', y)$. The filled area is $\C(x; Q_j)$, the box on the right is $(x+Q_K)\setminus Q_{j+K/2}$.}
\end{figure}
\begin{eqnarray*} \EE_1(x, M, K) &=& \Big \{ 0
\stackrel{Q_j}{\lrlong} x \, , x \textrm{ is } K\textrm{-regular} \and X_j^{\textrm{reg}} = M\Big \} \, , \\
\EE_2(x, x', y) &=& \Big \{ x' \lr y \off \C(x; Q_j) \Big \} \, ,\\
\EE_3 (x, x') &=& \Big \{ \C(x) \cap \C(x') = \emptyset \Big \} \, .
\end{eqnarray*}
In the following we sometimes abbreviate $\EE_1, \EE_2$ and $\EE_3$.

\begin{lemma} \label{sume1e2} There exists a constant $c>0$ such
that if $K>0$ is large enough then for any
$x\in \partial Q_j$ and any $x' \in (x+Q_K)\setminus Q_{j+K/2}$ we have that
$$ \sum _{y \in x + Q_L} \prob \big (\EE_1 \cap
\EE_2 \big ) \geq c L^2 \prob(\EE_1) \, .$$
\end{lemma}
\begin{proof} Note that $\EE_1$ can be determined by observing only the edges of $\C(x; Q_j)$. Hence we condition on $\C(x; Q_j)=A$ and get that
\begin{equation}\label{e1e2} \prob \big ( \EE_1 \cap
\EE_2 \big ) = \sum_{A \textrm{ admissible}} \prob \big (
\C(x; Q_j) = A )\cdot
\prob ( x' \lr y \off A \mid \C(x; Q_j) = A )
\, ,
\end{equation}
where by $A$ admissible, we mean $A$ in which $\EE_1$ holds and $\prob(\C(x; Q_j) = A)>0$. Since the event $\{x' \lr y \off A\}$ depends only on the
status of edges not touching $A$ we have that
$$ \prob ( x' \lr y \off
A \, \mid \, \C(x; Q_j) = A ) = \prob ( x' \lr y \off A) \, .$$
Continuing we write
\be\label{offon} \prob (x' \lr y \off A) = \prob (x' \lr
y) - \prob(x' \lr y \on A) \, .\ee
If $x' \lr y \on A$, then
there exists $z\in A$ such that $\{x' \lr z\} \circ \{z \lr y\}$.
This together with the $2$-point function estimate (\ref{tpt}) gives that
$$ \prob(x' \lr y \on A) \leq C \sum_{z \in A} |z-x'|^{2-d}
|z-y|^{2-d} \, .$$
We sum this over $y$ and get that
\be\label{midstep1} \sum_{y\in x+Q_L} \prob(x' \lr y \on A) \leq
CL^2 \sum_{z \in A} |z-x'|^{2-d}\, .\ee
We separate the sum dyadically over $z$ according to the scale of $z$'s distance
from $x'$ as follows. For a given $t\geq 0$ let
\[
A_t=A\cap\big(x'+\left(Q_{2^t}\setminus Q_{2^{t-1}}\right)\big)\,.
\]
With this notation we can write
$$ \sum_{z \in A} |z-x'|^{2-d} \leq C \sum_{t = \lceil\log (K/2)\rceil}^{\infty}
|A_t|2^{t(2-d)} \, ,$$ where we began the sum on $t$ from $\lceil\log (K/2)\rceil$
because if $z\in A$, then $z\in Q_j$ and hence $|z-x'|\geq K/2$ by our assumption on $x'$. By the same assumption, note that for any $s$ such that $s \geq K/2$ we have that
\be\label{trivbound} x'+Q_s \subset x+Q_{2s} \, .\ee
We now claim that
\be\label{eq:Atsmallsimp}
|A_t|<2^{4(t+1)}(t+1)^7\qquad\forall t\mbox{ such that }2^t\ge K/2\,.
\ee
Indeed, if $|A_t|\geq 2^{4(t+1)}(t+1)^7$ then directly from the definition of
$\Ss$ we have that $\prob( \Ss_{2^{t+1}}(x) \, |\, \C(x; Q_j)=A) = 0$ . This is what we termed in the discussion after Definition \ref{globalbadness} a ``simple''
bad configuration. However, $A$ is admissible whence $x$ is
$K$-regular, and we get a contradiction, hence (\ref{eq:Atsmallsimp}). Thus,
$$ \sum_{z \in A} |z-x'|^{2-d} \leq C \sum_{t = \lceil\log (K/2)\rceil}^\infty t^7 2^{t(6-d)} \leq C K^{6-d}\log^7 K \, .$$
We put this back into (\ref{midstep1}), and sum (\ref{offon}) over
$y$ using (\ref{tpt}). We get that
\begin{eqnarray*}
\lefteqn{\sum_{y\in x + Q_L} \prob (x' \lr y \off A) \geq}&&\\
&\qquad&\ge c\bigg(\sum _{y\in x+Q_L} |x'-y|^{2-d}\bigg) - C L^2 K^{6-d}\log^7 K
\ge \\
&&\ge L^2(c-CK^{6-d}\log^7 K)\, ,
\end{eqnarray*}
and so when $K$ is chosen large enough we have that
$$ \sum_{y\in x + Q_L} \prob (x' \lr y \off A) \geq c L^2 \, ,$$
and putting this back into (\ref{e1e2}) gives the assertion of the lemma.
\end{proof}

Our next step is the following easy estimate.

\begin{claim}\label{technical} Let $B \subset \Z^d$ be a set of vertices. Let $x'$ be a uniform random vertex chosen from a finite set $A$, then for any integer $s$ we have
$$ \E |(x'+Q_{s})\cap B| \leq \frac{|Q_s||B|}{|A|} \, .$$
\end{claim}
\begin{proof}
Indeed, for any $w \in Q_{s}$ we have that $\prob(x'+w \in B) \leq |B||A|^{-1}$.
\end{proof}

\begin{lemma} \label{sume1e2e3} There exists a constant $c>0$ and $K>0$ large enough such that for any $x\in \partial Q_j$ there exists $x' \in (x+Q_K)\setminus Q_{j+K/2}$ with
$$ \sum _{y \in x + Q_L} \prob \big (\EE_1 \cap
\EE_2 \cap \EE_3 \big ) \geq c L^2 \prob(\EE_1) \, .$$
\end{lemma}
\noindent{\em Remark.} The statement in fact holds for any $x'\in (x+Q_K)\setminus Q_{j+K/2}$ but proving this takes an extra effort. We  only require one such $x'$ and choosing $x'$ at random simplifies the proof of this lemma significantly.
\begin{proof} We take $x'$ to be a uniform random vertex in $(x+K)+Q_{K/2}$ and prove that
$$\E_{x'} \Big [ \sum _{y \in x + Q_L} \prob \big (\EE_1 \cap
\EE_2 \cap \EE_3 \big ) \Big ] \geq c L^2 \prob(\EE_1) \, ,$$
and it follows that there exists $x'$ such that the assertion of the lemma holds.

\begin{figure}
\input{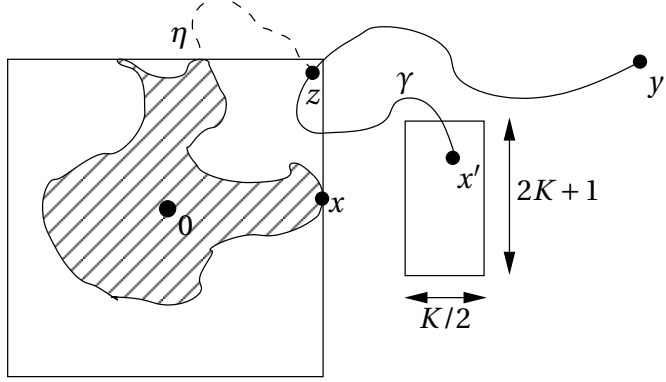}
\caption{\label{cap:e1e2note3} The event $\EE_1\cap \EE_2\cap \neg \EE_3$. The solid line between $x'$ and $y$ is the path $\gamma$ and the dashed line is the path $\eta$.}
\end{figure}

For any $x'\in \partial Q_j$, assume that
$\EE_1\cap \EE_2\cap \neg \EE_3$ occurs. We claim that in this case
there exists a vertex $z$ such that the event
\be\label{eq:BKR}
\{ \EE_1 \cap 0 \lr z \} \circ \{ x' \lr z \} \circ \{ z \lr y
\} \, ,
\ee occurs. Indeed, let $\gamma$ be an open path between
$x'$ and $y$ which avoids $\C(x; Q_j)$. Since we assume that $\neg
\EE_3$ occurs (that is, we assume $x \lr x'$) there must exists an open path
$\eta$ connecting a vertex on $\C(0;Q_j)$ to a vertex on
$\gamma$ such that, considered as sets of edges, $\eta \cap
(\C(0; Q_j) \cup \gamma) = \emptyset$. Denote by $z$ the end vertex
of $\eta$ ($z$ is a vertex on the path $\gamma$).
To verify (\ref{eq:BKR}) we check that the three events can be
verified with disjoint set of edges. Indeed, to verify $\EE_1
\cap \{0\lr z\}$ it suffices to observe the edges of $\C(x; Q_j)$
and $\eta$. Note that ``the edges of $\C(x;Q_j)$'' means all edges
needed to calculate $\C(x;Q_j)$ precisely, i.e.\ all open edges inside
the cluster and all closed edges defining its boundary in $Q_j$. To verify $\{ x' \lr z \}$ we observe the edges of
$\gamma$ up to $z$, and to verify $\{ z \lr y \}$ we observe the edges
of $\gamma$ from $z$ to $y$. See Figure \ref{cap:e1e2note3}. The BK-Reimer inequality
gives that
$$
\prob(\EE_1\cap \EE_2\cap \neg \EE_3) \leq \sum _{z} \prob(\EE_1 \cap 0 \lr z)
\prob(x' \lr z) \prob(z \lr y) \, .
$$
We sum over $y$ and use the $2$-point function estimate
(\ref{tpt}) to get that \be\label{midstep2} \sum_{y \in x + Q_L}
\prob(\EE_1\cap \EE_2\cap \neg \EE_3) \leq C L^2 \sum_{z} \prob(\EE_1 \cap
0 \lr z) |z-x'|^{2-d} \, .\ee



To sum over $z$, as in the previous lemma, we separate
the sum over $z$ according to the scale of the distance of $z$ from
$x'$ and condition on $\C(x; Q_j)$. Define
\[
B_t(x')=\C(0)\cap\Big(x'+\left(Q_{2^t}\setminus
Q_{2^{t-1}}\right)\Big)\,.
\]
We get that (\ref{midstep2}) is bounded above by
\begin{equation}\label{eq:mainsum}
CL^2 \prob(\EE_1) \sum_{t\geq 1} \sum_{A \textrm{ admissible}} \E\big(|B_t(x')|\cdot\mathbf{1}_{\C(x; Q_j) = A} \big)2^{t(2-d)}\, ,
\end{equation}
where again by $A$ admissible, we mean $A$ in which $\EE_1$ holds and
$\prob(\C(x; Q_j) = A)>0$. It is at this point that we finally use the
full power of our definition of bad vertices. Assume $K$ is a power of
two, and put $t_0 = \log (K/2)$. We first sum (\ref{eq:mainsum}) over
$t > t_0$. For such $t$, as in (\ref{trivbound}) we have that $|B_t(x')| \leq |\C(0)\cap (x+Q_{2^{t+1}})|$ for all $x'$ and we split the estimate according to whether $\Ss_{2^{t+1}}(x)$ occurs. If it does occur, then by definition of $\Ss_{2^{t+1}}(x)$ and the fact that $x\lr 0$ we have that $|\C(0)\cap (x+Q_{2^{t+1}})| \leq C 2^{4t} t^7$, whence
\be\label{eq:typical}
\E\big(|B_t(x')|\cdot \mathbf{1}_{\Ss_{2^{t+1}}(x)} \;\big|\; \C(x;
Q_j) = A \big) \leq C 2^{4t} t^7 \, .
\ee
On the other hand, since $x$ is $K$-regular it is not $2^{t+1}$-bad
for $t > t_0$ so by Definition \ref{globalbadness},
\begin{multline*}
\E\left(|B_t(x')|\mathbf{1}_{\neg\Ss_{2^{t+1}}(x)} \;\Big|\;
  \C(x; Q_j) = A \right) \leq \\
\leq |Q_{2^t}|\cdot\prob\left(\neg\Ss_{2^{t+1}}(x)\;\Big|\;
  \C(x; Q_j) = A \right) \le
C2^{td} e^{-t^2} \, .
\end{multline*}
This is negligible with respect to (\ref{eq:typical}) and we learn that
\be \label{larget} \E\left(|B_t(x')|\;\Big|\; \C(x; Q_j) = A \right) \leq C 2^{4t} t^7 \, ,\ee for $t > t_0$ and all $x'$.

Next we sum over $t \leq t_0$ and here is where we use the fact that
$x'$ is randomized. We perform a split similar to before, but consider
a box of size $2^{t_0+1}=K$ rather than $2^{t+1}$. Namely, if
$\Ss_K$ occurs then $|\C(0)\cap (x+Q_K)| \leq (2K)^4 \log^7 (2K)$ and by Claim \ref{technical} we have that for any $t \leq t_0$
$$ \E_{x'} |\C(0) \cap (x'+Q_{2^t})| \leq \frac{C 2^{td} K^4 \log^7 K}{K^d} \, ,$$
where $B$ from Claim \ref{technical} was taken to be $\C(0)\cap (x+Q_K)$. We deduce that
$$ \E_{x'} \E\left(|B_t(x')|\mathbf{1}_{\Ss_K(x)} \;\Big|\; \C(x; Q_j) = A \right) \leq \frac{C 2^{td} K^4 \log^7 K}{K^d} \, .$$
The case of $\neg\Ss$ is as before. Since $x$ is $K$-regular, for all
$t \leq t_0$ we have that
$$ \E\left(|B_t(x')|\mathbf{1}_{ \neg\Ss_K(x)} \;\Big|\;
\C(x; Q_j) = A \right) \leq C2^{td} e^{-t_0^2} \, ,
$$
which is again negligible, and we deduce that for any $t \leq t_0$,
$$ \E_{x'} \E\left(|B_t(x')| \;\Big|\; \C(x; Q_j) = A \right) \leq \frac{C 2^{td} K^4 \log^7 K}{K^d} \, .$$
We put this together with (\ref{larget}) and get that for any admissible $A$
\begin{align*} \E _{x'} \sum_{t\geq 1} \E\left(|B_t(x')|\;\Big|\;
  \C(x; Q_j) = A \right) 2^{t(2-d)} &\leq C K^{4-d} \log^7 K \sum_{t
  \leq t_0} 2^{2t} + C \sum _{t > t_0} 2^{t(6-d)} t^7 \\
&\leq C K^{6-d} \log ^7 K
\end{align*}
(recall that $t_0=\log(K/2)$). We put this into (\ref{eq:mainsum}) and that into (\ref{midstep2}) and conclude that
$$ \E_{x'} \sum_{y \in x + Q_L} \prob(\EE_1\cap \EE_2\cap \neg \EE_3) \leq CL^2 \prob(\EE_1) K^{6-d} \log^7 K \, .$$
We now apply Lemma \ref{sume1e2} and choose $K$ large enough and we are done.
\end{proof}

We are now ready to prove Lemma \ref{firstmoment} and
\ref{secondmoment}.
\medskip

\noindent {\bf Proof of Lemma \ref{firstmoment}.}
The lemma will follow directly from Lemma \ref{sume1e2e3} once we show that
\begin{multline}\label{eq:admisEEE}
\prob((x,y)\mbox{ are }(j,L,K)\mbox{-admissible and
}X_j^\textrm{reg}=M) \ge\\
\ge c(K)\prob\big (\EE_1(x,M,K)\cap\EE_2(x,x',y)\cap\EE_3(x,x')\big)
\end{multline}
for all $x$ and $y$ and $x'$ chosen according to Lemma \ref{sume1e2e3} --- summing (\ref{eq:admisEEE}) over $y$ gives the
$L^2$ factor, by Lemma \ref{sume1e2e3}, and the sum over $x$ obviously
gives a factor of $M$. So we only need to show (\ref{eq:admisEEE}).


To show (\ref{eq:admisEEE}) we use a local modification
argument as follows. Let $x$, $x'$ and $y$ satisfy $\EE_1\cap\EE_2\cap\EE_3$.
Write $\gamma$ for
the path connecting $x'$ to $y$ which avoids $\C(x)$. Consider
the edges in
\begin{figure}
\input{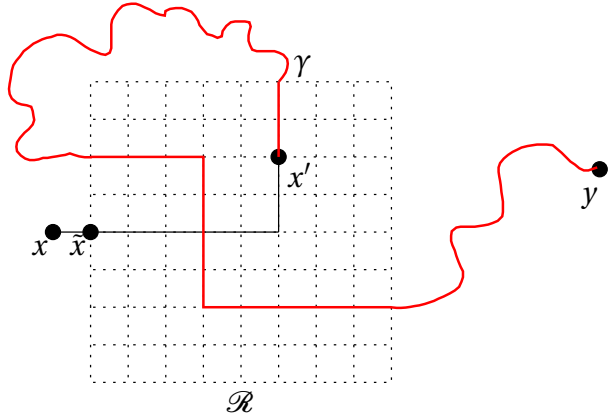}
\caption{\label{cap:rect} The local modification is performed in $\RR$. The thick red path connecting $x'$ to $y$ is $\gamma$.}
\end{figure}
\[\RR = (x+Q_K)\setminus Q_{j+1}\]
so that $x' \in \RR$, see Figure \ref{cap:rect}. Let us now apply the
following modification. Close all the edges in $\RR$ except edges
belonging to $\gamma$, and open the edges of an arbitrary path in
$\RR$ starting at $\tilde{x}$ (recall that $\tilde{x}$ is the neighbor
of $x$ outside $Q_j$) and ending at $x'$ (the black path in
Figure \ref{cap:rect}). Now open the edge $(x,\tilde{x})$.
In the new configuration, $(x,y)$ is $(j,L,K)$-admissible and $X_j^\textrm{reg}$
is still equal to $M$. Indeed, $X_j^\textrm{reg}$ depends only on what
happens inside $Q_j$ and we changed nothing there. For the same reason
$x$ remains $K$-regular. The fact that in the original configuration
$\C(x) \cap \C(x') = \emptyset$ ensures that the edge $(x,\tilde{x})$ is
pivotal for $0 \lr y$ in the modified configuration. Hence all
conditions for admissibility are satisfied.

Now, in this modification we changed the status of at most $(2K)^d$
edges, which means that (\ref{eq:admisEEE}) holds with
$c(K)=\left(\frac{1}{2}\min(p_c,1-p_c)\right)^{(2K)^d}$, and the lemma is proved.
\qed
\medskip

\noindent{\bf Proof of Lemma \ref{secondmoment}.}
If $(x_1,y_1)$
and $(x_2,y_2)$ are both $(j, K, L)$-admissible and
$\{X_j^{\textrm{reg}} = M \}$ holds, then one of following three
events occur:
\begin{enumerate}
\item $x_1=x_2$, $y_1=y_2$ and
$$ \{\EE_1(x_1)\} \circ \{x_1 \lr y_1 \} \, ,$$

\item $x_1=x_2$, $y_1\ne y_2$ but both in $x+Q_L$ and there exists some
$z$ such that
$$\{\EE_1(x_1)\} \circ \{x_1 \lr z \} \circ \{ z \lr y_1 \} \circ \{z \lr y_2 \} \, ,$$

\item $x_1 \neq x_2$, $y_1 \neq y_2$, $y_i \in x_i + Q_L$ and
$$ \{ \EE_1(x_1), \EE_1(x_2) \} \circ \{ x_1 \lr y_1 \} \circ \{x_2 \lr y_2 \} \, .$$
\end{enumerate}
To see this, first note that if $x_1=x_2$ and $y_1\neq y_2$, then one
may consider the cluster $\C$ of all vertices connected to $0$ only through
$(x_1,\tilde{x}_1)$. By the definition of admissibility it contains both
$y_1$ and $y_2$ and then one may define $z$ to be the triple point of
$\tilde{x}_1$, $y_1$ and $y_2$ in $\C$ in the usual way. Since
$\tilde{x}_1\not\in Q_j$ we see that $\C\cap \C(0;Q_j)=\emptyset$ and
hence the edges needed to define $\C(0;Q_j)$ --- which are enough to
prove that $\EE_1(x_1)$ occurred --- are disjoint from those defining
the three paths between $x_1$ and $z$, $z$ and $y_1$ and $z$ and
$y_2$. This shows (ii).

Assume now that $x_1 \neq x_2$, and define $\C_i$ to be the cluster of
vertices connected to $0$ only through $x_i$. Both $\C_1$ and $\C_2$
are non-empty because $y_i\in\C_i$ and they are disjoint, because if
$z\in\C_1\cap\C_2$ then taking a simple open path from $z$ to $0$ and examining
which of the edges $(x_i,\tilde{x}_i)$ it passes first, it is clear
that it does not need to pass through the other, contradicting the
definition of $\C_i$. Thus $\C_1\cap\C_2=\emptyset$ and we can choose
open paths demonstrating that $x_i\lr y_i$ which are both disjoint and
disjoint from $\C(0;Q_j)$. This shows (iii) and the whole trichotomy.

We get that
$$ \E Y^2 {\bf
1}_{\{X^{\textrm{reg}}=M \}} \leq S_1 + S_2 + S_3 \, ,$$ where
\begin{align*} S_1 &= \sum_{\substack{ x \in \partial Q_j  \\ y\in x + Q_L}} \prob(\EE_1(x))\prob(x \lr y) \, ,\\
 S_2 &= \sum_{\substack{ x \in \partial Q_j \\ y_1, y_2 \in x+Q_L}}
\prob (\EE_1(x)) \sum _{z} \prob(x \lr z) \prob(z \lr y_1) \prob(z
\lr y_2) \, ,\\
S_3 &= \sum _{\substack{x_1, x_2\in \partial Q_j \\ y_i \in x_i+Q_L}} \prob( \EE_1(x_1) \cap \EE_1(x_2) ) \prob(x_1 \lr y_1)
\prob(x_2 \lr y_2) \, .
\end{align*}
Using (\ref{tpt}) we easily estimate $S_1$
by
$$ S_1 \leq C L^2 \sum _{x \in \partial Q_j} \prob(\EE_1(x)) =
CML^2 \prob (X^{\textrm{reg}}=M) \, ,$$ where the last equality
follows by definition of $\EE_1$. To estimate $S_2$ we sum over $y_1$,
$y_2$ and $z$ as in Lemma \ref{lem:xyz} and get a term of $L^6$ so
$$ S_2 \leq C M L^6 \prob (X^{\textrm{reg}}=M) \, .$$
Finally we use the $2$-point estimate (\ref{tpt}) to estimate $S_3$
and get
\begin{align*} S_3
&\leq \sum_{x_1, x_2\in \partial Q_j} \prob( \EE_1(x) \cap \EE_1(x') ) \sum _{y_i \in x_i+Q_L} |x_1-y_1|^{2-d} |x_2-y_2|^{2-d}
\\ &\leq CL^4 \sum_{x_1, x_2\in \partial Q_j} \prob( \EE_1(x) \cap \EE_2(x') ) = C M^2 L^4 \prob (X^{\textrm{reg}}=M) \, .\end{align*}
We conclude that
$$ \E Y^2 {\bf 1}_{\{X^{\textrm{reg}}=M \}} \leq
C M^2 L^4\prob(X^\textrm{reg}=M) \, ,$$ since $M \geq L^2/2$.
\qed

\section{Multiple arms.}\label{sec:multiple} The upper bound of
$r^{-2\ell}$ follows immediately from the BK inequality and so the
main effort in this chapter is to prove the lower bound. To that aim
we require an ``inverse''-BK inequality. Our proof follows the
standard proof of the BK inequality. Roughly, it starts with
two identical copies of the graph, with one event on each copy, and
then merging edges, and showing that the probability decreases with
each merge. We will perform the same analysis on two copies of
$Q_{2r}$ but will only merge the edges of $Q_r$, and estimate how much
is lost in each merge operation.

We begin by describing the setting, using the notation of \cite{G}.
Let $m >0$ be an integer and let $\Omega$ be the set of all $0$-$1$ vectors of length $m$. Let $\prob$ be a product probability measure on $\Omega$ with density $p_i$ on the $i$-th coordinate, that is
$$ \prob(\omega) = \prod_{i=1}^m [\omega_i p_i + (1-\omega_i)(1-p_i)] \qquad \forall \,\, \omega \in \Omega \, .$$
Now, Let $(\Omega, \prob)$ and $(\widetilde{\Omega}, \widetilde{\prob})$ be two copies of $(\Omega, \prob)$ and write $(\Omega \times \widetilde{\Omega}, \prob_\otimes)$ for the product space where $\prob_\otimes = \prob \times \widetilde{\prob}$ is the product measure. Given two increasing events $\A$ and $\B$ in $\Omega$ write $\A'\subset \Omega \times \widetilde{\Omega}$ for
$$ \A' = \big \{ (\omega,\widetilde{\omega}) \in \Omega \times \widetilde{\Omega} : \omega \in \A \big \} \, ,$$
and $\B'_0\subset \Omega \times \widetilde{\Omega}$ for
$$ \B'_0 = \big \{ (\omega,\widetilde{\omega}) \in \Omega \times \widetilde{\Omega} : \widetilde{\omega} \in \B \} \, .$$
For each $k \in \{1, \ldots , m\}$ write $\B'_k \subset \Omega \times \widetilde{\Omega}$ for
$$ \B'_k = \big \{ (\omega,\widetilde{\omega}) \in \Omega \times
\widetilde{\Omega} : (\omega_1, \ldots, \omega_{k}, \widetilde{\omega}_{k+1}, \ldots,
\widetilde{\omega}_m) \in \B \big \} \, .$$
In words, $\B'_k$ is the event after merging the first $k$ edges.
Note that $\prob_{\otimes}(\A'\circ \B'_0)=\prob(\A)\prob(\B)$ and
that $\prob_{\otimes}(\A' \circ \B'_m) = \prob(\A \circ \B)$. The BK
inequality follows immediately once one shows that for any $k$ we have
$\prob_{\otimes}(\A'\circ \B'_{k}) \leq \prob_{\otimes}(\A'\circ
\B'_{k-1})$. The proof of this fact can be found in \cite{G}, but we
do not need this here. What we will need is
\be\label{inversebk} \prob_{\otimes} ( \A' \circ \B'_k ) = \prob_{\otimes}(\A' \circ \B'_0) - \sum_{i=1}^k \big [ \prob_{\otimes}(\A' \circ \B'_{i-1}) - \prob_{\otimes} ( \A' \circ \B'_{i}) \big ] \, .\ee

In our setting, let $m$ be the number of edges which have at least one end in $Q_{2r}$ and take $(\Omega, \prob)$ to be the usual Bernoulli percolation measure on these edges. Again, let $(\Omega, \prob)$ and $(\widetilde{\Omega}, \widetilde{\prob})$ be two copies of $(\Omega, \prob)$ and $(\Omega \times \widetilde{\Omega}, \prob_{\otimes})$ to be the product measure. Let $e_1, \ldots, e_m$ and $\widetilde{e}_1, \ldots, \widetilde{e}_m$ be the edges corresponding to $\omega_1, \ldots, \omega_m$ and $\widetilde{\omega}_1, \ldots, \widetilde{\omega}_m$, respectively. We assume that that the edges are ordered in such a way that there exists a number $k<m$ such that all the edges $e_1,\ldots,e_k$ and $\widetilde{e}_1, \ldots, \widetilde{e}_k$ have at least one end vertices in $Q_{r}$ and the rest of the edges have both endpoints not in $Q_r$.

Given $y_1, \ldots, y_\ell \in \Z^d$ as in the statement of Theorem \ref{multiplearms}, with constant $K$ to be chosen later, we define the events $\A$ and $\B$ by
\begin{eqnarray*} \A &=& \{y_1 \lr \partial Q_{2r}\} \circ \cdots \circ \{y_{\ell-1} \lr \partial Q_{2r}\} \, , \\
\B &=& \{y_\ell \lr \partial Q_{2r}\}\, . \end{eqnarray*}

\begin{lemma} \label{pivotalmerge} Assume the setting of Theorem \ref{multiplearms}. Let $i \leq k$ and write $e_i=(z,z')$ for the corresponding edge in $Q_r$. Then
$$ \prob  ( \A' \circ \B'_{i-1}) - \prob( \A' \circ \B'_{i} \big ) \leq C r^{-2\ell} \sum_{j=1}^{\ell-1}\prob(y_j \lr z) \prob(y_\ell \lr z) \, ,$$
where $C>0$ is a constant that depends on $\ell, d$ and the lattice chosen.
\end{lemma}
\begin{proof} If $(\omega,\widetilde{\omega}) \in \A' \circ \B'_{i-1} \setminus \A' \circ \B'_{i}$ then we must have that $\widetilde{\omega}_i=1$. Consider the sets
$$ \mathcal{D} = \big \{ (\omega,\widetilde{\omega}) \in \A' \circ \B'_{i-1} \setminus \A' \circ \B'_{i} \, : \, \widetilde{\omega}_i=1, \omega_i=0 \big \} \, ,$$
and
$$ \mathcal{D}' = \big \{ (\omega,\widetilde{\omega}) \in \A' \circ \B'_{i} \setminus \A' \circ \B'_{i-1} \, : \, \widetilde{\omega}_i=0, \omega_i=1 \big \} \, .$$
A moment's reflection shows that the map $\varphi$ which exchanges the
values of $\omega_i$ and $\widetilde{\omega}_i$ is a one-to-one measure preserving map from
$\mathcal{D}$ onto $\mathcal{D}'$ --- indeed, both are
characterized by the condition that any choice of the two sets $U$ and
$V$ in the definition of $\circ$ satisfies $\widetilde{e}_i\in V$ or
$e_i\in V$, respectively. We deduce that
$$ \prob  ( \A' \circ \B'_{i-1}) - \prob( \A' \circ \B'_{i} \big )
\leq \prob \Big ( \big \{ (\omega,\widetilde{\omega}) \in
      \A' \circ \B'_{i-1}\setminus \A'\circ\B'_i \, : \,
  \widetilde{\omega}_i=1, \omega_i=1 \big \} \Big ) \, .
$$
If $\A'\circ \B'_{i-1}\setminus\A'\circ\B'_i$ occurs and $\widetilde{\omega}_i=\omega_i=1$, then $\A'$ must use $e_i$ and $\B'_{i-1}$ must use $\widetilde{e}_i$. This implies that for some $j \in \{1, \ldots, \ell-1\}$ we have that the events
\begin{itemize}
\item $\{y_n \lr \partial Q_{2r}\}$ for all $n \in \{1, \ldots, \ell-1\} \setminus \{j\}$ using the edges $e_1,\ldots, e_m$,
\item $\{y_j \lr z\} \cup \{y_j \lr z'\}$ using the edges $e_1,\ldots, e_m$,
\item $\{y_\ell \lr z\} \cup \{y_\ell \lr z'\}$ using the edges $e_1, \ldots, e_{i-1}, \widetilde{e}_{i}, \ldots \widetilde{e}_{i+1},  \ldots, \widetilde{e}_m$,
\item $\{z \lr \partial Q_{2r}\} \cup \{z' \lr \partial Q_{2r}\}$ using the edges $e_1,\ldots, e_m$,
\item $\{z' \lr \partial Q_{2r}\}\cup \{z' \lr \partial Q_{2r}\}$ using the edges $e_1, \ldots, e_{i-1}, \widetilde{e}_{i}, \ldots \widetilde{e}_{i+1},  \ldots, \widetilde{e}_m$,
\end{itemize}
occur disjointly. By the BK inequality we get that
\begin{align*} \lefteqn{\prob \big ( \A' \circ \B'_{i-1} \setminus \A' \circ \B'_{i} \big ) \leq}&&& \\ &&& 16 \sum _{j \leq \ell-1} \Big [ \prod _{\substack{n \leq \ell-1 \\ n \neq j}}\prob(y_n \lr \partial Q_{2r}) \Big ] \prob(y_j \lr z) \prob(y_\ell \lr z) \prob(z \lr \partial Q_{2r})^2 \, . \end{align*}
We now use Theorem \ref{mainthm} to conclude the proof of the lemma.
\end{proof}

We are now ready to prove Theorem \ref{multiplearms}.

\begin{proof}[Proof of Theorem \ref{multiplearms}] We prove the claim
  by induction on $\ell$. The case $\ell=1$ is precisely Theorem
  \ref{mainthm} so we may assume $\ell\geq 2$. Recall the definition
  on $k$ and the events $\A$ and $\B$ from above. By Lemma
  \ref{pivotalmerge}, (\ref{inversebk}) and the two-point function
  estimate (\ref{tpt}) we get that
$$ \prob_{\otimes} (\A' \circ \B'_k) \geq \prob(\A)\prob(\B) - C r^{-2\ell} \sum_{j=1}^{\ell-1}\sum_{z} |y_j -z |^{2-d} |y_\ell-z|^{2-d} \, .$$
We sum this over $z$ and use the induction hypothesis to get that
$$ \prob_{\otimes} (\A' \circ \B'_k) \geq cr^{-2\ell} - C r^{-2\ell}
\ell K^{4-d} \, .$$
Hence if we choose $K$ large enough (depending on $\ell$) we get that $\prob_{\otimes} (\A' \circ \B'_k) \geq c r^{-2\ell}$. Now, since the edges $\{e_i\}_{i\leq k}$ are all the edges which have an endpoint in $Q_r$ the event $\A' \circ \B'_k$ implies that for all $j \leq \ell$ the events $y_j \lr \partial Q_r$ occur disjointly using the edges $e_1, \ldots, e_k$. This concludes our proof.
\end{proof}


\vspace{.05 in}\noindent
{\bf Gady Kozma}: \texttt{gady.kozma(at)weizmann.ac.il} \\
The Weizmann Institute of Science, \\
Rehovot POB 76100, \\
Israel.

\medskip \noindent
{\bf Asaf Nachmias}: \texttt{asafn(at)microsoft.com} \\
Microsoft Research,
One Microsoft way,\\
Redmond, WA 98052-6399, USA.


\begin{thebibliography}{99}

\bibitem{AS} Aharony A. and Stauffer D. (1991), Introduction To Percolation Theory, CRC Press.

\bibitem{A} Aizenman M. (1997), On the number of incipient spanning
  clusters. {\em Nuclear Phys. B} {\bf 485}, no. 3, 551--582.

\bibitem{AB} Aizenman M. and Barsky D. J. (1987), Sharpness of the
  phase transition in percolation models. {\em Commun. Math. Phys.} {\bf 108},
  no. 3, 489--526.


\bibitem{AN} Aizenman M. and Newman C. M. (1984) Tree graph
  inequalities and critical behavior in percolation models. {\em
    J. Statist. Phys.} {\bf 36}, no. 1-2, 107--143.

\bibitem{AthNey} Athreya K. B. and Ney P. E. (1927) Branching processes. Die Grundlehren der mathematischen Wissenschaften, Band 196. Springer-Verlag, New York-Heidelberg.

%
%
%
%




%

\bibitem{BA} Barsky D. J. and Aizenman M. (1991), Percolation critical
  exponents under the triangle condition. {\em Ann. Probab.} {\bf 19}, no. 4,
  1520--1536.


\bibitem{BF87} van den Berg J. and Fiebig U. (1987),
On a combinatorial conjecture concerning disjoint occurrences of events.
{\em Ann. Probab.} {\bf 15}, no. 1, 354--374.

\bibitem{vdBK} van den Berg J. and Kesten H. (1985), Inequalities with
  applications to percolation and reliability. {\em J. Appl. Probab.}, {\bf 22},
  556--569.

\bibitem{BR} Bollobás, B. and Riordan, O. (2006),
  Percolation. Cambridge University Press, New York.


\bibitem{BCR} Borgs C., Chayes J. T. and Randall D. (1999) The van den
  Berg-Kesten-Reimer inequality: a review.  Perplexing problems in
  probability,  159--173, Progr. Probab., 44, Birkhauser Boston, Boston, MA,
  1999.


\bibitem{BS} Brydges, D. and Spencer, T. (1985) Self-avoiding walk in
  $5$ or more dimensions. {\em Commun. Math. Phys.},  {\bf 97},
  no. 1-2, 125--148.

\bibitem{BK} Burton R. M. and Keane M. (1989),
Density and uniqueness in percolation.
{\em Commun. Math. Phys.} {\bf 121}, no. 3, 501--505.

\bibitem{CC} Chayes J. T. and Chayes L. (1987) On the upper critical
  dimension of Bernoulli percolation. {\em Commun. Math. Phys.} {\bf 113},
  no. 1, 27--48.



\bibitem{D} Durrett, R. (1996), {\em Probability: Theory and Examples}, Second
edition. Duxbury Press, Belmont, California.

\bibitem{ER} Erd\H{o}s P. and R\'enyi A. (1960),
On the evolution of random graphs, {\em Magyar Tud.\ Akad.\ Mat.\
Kutat\'o Int. K\H{o}zl.} {\bf 5}, 17--61.

\bibitem{G} Grimmett G. (1999), Percolation. Second edition. Grundlehren der Mathematischen Wissenschaften, 321. Springer-Verlag, Berlin.

%

\bibitem{Ha} Hara T. (2008), Decay of Correlations in
  Nearest-Neighbour Self-Avoiding Walk, Percolation, Lattice Trees and
  Animals, {\em Ann. Probab.}, {\bf 36}, no. 2, 530--593.

\bibitem{HaHS} Hara T., van der Hofstad R. and Slade G. (2003),
  Critical two-point functions and the lace expansion for spread-out
  high-dimensional percolation and related models, {\em Ann. Probab.},
  {\bf 31}, no. 1, 349--408.

\bibitem{HaS0} Hara T. and Slade G. (1990), Mean-field critical
  behaviour for percolation in high dimensions. {\em
  Commun. Math. Phys.}, {\bf 128}, no. 2, 333--391.

\bibitem{HaS} Hara T. and Slade G. (2000), The scaling limit of the
  incipient infinite cluster in high-dimensional
  percolation. I. Critical exponents. {\em J. Statist. Phys.}, {\bf
  99}, no. 5-6, 1075--1168.

\bibitem{Harris} Harris T. E. (1960), A lower bound for the critical probability in a certain percolation process. {\em Proc. Cambridge Philos. Soc.} {\bf 56}, 13--20.


\bibitem{HHS07} Heydenreich M., van der Hofstad R. and Sakai
  A. (2008), Mean-field behavior for long- and finite range Ising
  model, percolation and
  self-avoiding walk. {\em J. Statist. Phys.} {\bf 132}, no. 6, 1001--1049.

\bibitem{HHS} van der Hofstad R., den Hollander F. and Slade G. (2002),
  Construction of the incipient infinite cluster for spread-out oriented
  percolation above $4 + 1$ dimensions. {\em Commun. Math. Phys.}, {\bf 231},
  435--461.

\bibitem{HHS1} van der Hofstad R., den Hollander F. and Slade
  G. (2007) The survival probability for critical spread-out oriented
  percolation above $4+1$ dimensions. I. Induction. {\em
  Probab. Theory Relat. Fields.} {\bf 138}, no. 3-4, 363--389.

\bibitem{HHS2} van der Hofstad R., den Hollander F. and Slade
  G. (2007) The survival probability for critical spread-out oriented
  percolation above $4+1$ dimensions.  II. Expansion. {\em
  Ann. Inst. H. Poincaré Probab. Statist.} {\bf 43}, no. 5, 509--570.

%



\bibitem{K80} Kesten H. (1980), The critical probability of bond percolation on the square lattice equals $\frac{1}{2}$. {\em Commun. Math. Phys.} {\bf 74}, no. 1, 41--59.

\bibitem{K82} Kesten H. (1982), Percolation theory for
  mathematicians. Progress in Probability and Statistics, 2. Birkhauser,
  Boston, Mass.




\bibitem{Ko} Kolmogorov A. N. (1938), Zur Lösung einer biologischen
  Aufgabe [German: On the solution of a problem in biology]. {\em Izv. NII Matem. Mekh. Tomskogo Univ.} {\bf 2},
7--12.

\bibitem{KN} Kozma G. and Nachmias A. (2009),
The Alexander-Orbach conjecture holds in high dimensions, {\em
  Invent. Math.}, {\bf 178}, no. 3, 635--654.

\bibitem{Koz1} Koz\'akov\'a, Iva (2008), Critical percolation of free
  product of groups.  {\em Internat. J. Algebra Comput.} {\bf 18},
  no. 4, 683--704.

\bibitem{Koz2} Koz\'akov\'a, Iva (2008 preprint), Critical percolation on
  Cayley graphs of groups acting on trees, \url{http://arxiv.org/abs/0801.4153}


\bibitem{LSW} Lawler G., Schramm O. and Werner W. (2002), One-arm exponent for critical 2D percolation. {\em Electron. J.
Probab.} {\bf 7}, no. 2.


\bibitem{M} Menshikov, M. V. (1986), Coincidence of critical points in
  percolation problems. (Russian) {\em Dokl. Akad. Nauk SSSR} {\bf 288},
  no. 6, 1308--1311. English translation in: {\em Soviet Math. Dokl.}
  {\bf 33}, no. 3, 856--859.





\bibitem{N87} Nguyen, B. G. (1987), Gap exponents for percolation
  processes with triangle condition. {\em J. Statist. Phys.} {\bf 49},
  no. 1-2, 235--243.



\bibitem{R00} Reimer, D. (2000) Proof of the van den Berg-Kesten conjecture.
{\em Combin. Probab. Comput.} {\bf 9}, no. 1, 27--32.

\bibitem{S} Sakai, A. (2004) Mean-field behavior for the survival
  probability and the percolation point-to-surface connectivity. {\em
  J. Statist. Phys.} {\bf 117}, no. 1-2, 111--130. Erratum: {\em
  J. Statist. Phys.} {\bf 119} (2005), no. 1-2, 447--448.

\bibitem{Sc} Schonmann R. H. (2001), Multiplicity of phase transitions and mean-field criticality on
highly non-amenable graphs, {\em Commun. Math. Phys.} {\bf 219}, no. 2,
271-322.


\bibitem{Sl} Slade G. (2006), The lace expansion and its applications. Lectures from the 34th Summer School on Probability Theory held in Saint-Flour, July 6--24, 2004.

\bibitem{Sm} Smirnov S. (2001 preprint), Critical percolation in the
  plane. Available at
  \url{http://www.unige.ch/~smirnov/papers/percol.ps}


\bibitem{SW} Smirnov, S. and Werner, W. (2001),
Critical exponents for two-dimensional percolation.
{\em Math. Res. Lett.} {\bf 8}, no. 5-6, 729--744.


\end{thebibliography}
\end{document}